\documentclass[12pt,reqno,openbib,runningheads,a4paper]{amsart}%
\usepackage{graphicx}
\usepackage{amssymb}
\usepackage{amsfonts}
\usepackage{amsmath}
\usepackage{graphicx}
\usepackage{lineno}
\usepackage[authoryear]{natbib}
\usepackage[dvipsnames]{xcolor}
\usepackage{epsf}
\usepackage{lscape}
\usepackage[legalpaper,bookmarks=true,colorlinks=true,linkcolor=blue,citecolor=blue]%
{hyperref}%
\providecommand{\U}[1]{\protect\rule{.1in}{.1in}}
\providecommand{\U}[1]{\protect\rule{.1in}{.1in}}
\newtheorem{theorem}{Theorem}[section]
\newtheorem{corollary}{Corollary}[section]

\newtheorem{lemma}{Lemma}[section]

\newtheorem{remark}{Remark}[section]

\makeatletter
\renewcommand{\@biblabel}[1]{}
\makeatother
\begin{document}

\begin{center}
{\Large \textbf{Tail product-limit process for truncated data with application
to extreme value index estimation}}

\medskip\medskip

{\large Souad Benchaira, Djamel Meraghni, Abdelhakim Necir}$^{\ast}$\medskip

{\small \textit{Laboratory of Applied Mathematics, Mohamed Khider University,
Biskra, Algeria}}\medskip\medskip%
\[
\]

\end{center}

\noindent\textbf{Abstract}\medskip

\noindent A weighted Gaussian approximation to tail product-limit process for
Pareto-like distributions of randomly right-truncated data is provided and a
new consistent and asymptotically normal estimator of the extreme value index
is derived. A simulation study is carried out to evaluate the finite sample
behavior of the proposed estimator.\medskip\medskip

\medskip


\vfill

\vfill

\noindent{\small $^{\text{*}}$Corresponding author:
\texttt{necirabdelhakim@yahoo.fr} \newline}

\noindent{\small \textit{E-mail addresses:}\newline%
\texttt{benchaira.s@hotmail.fr} (S.~Benchaira)\newline%
\texttt{djmeraghni@yahoo.com} (D.~Meraghni)}

\section{\textbf{Introduction\label{sec1}}}

\noindent Let $\left(  \mathbf{X}_{i},\mathbf{Y}_{i}\right)  ,$ $1\leq i\leq
N$ be a sample of size $N\geq1$ from a couple $\left(  \mathbf{X}%
,\mathbf{Y}\right)  $ of independent random variables (rv's) defined over some
probability space $\left(  \Omega,\mathcal{A},\mathbf{P}\right)  ,$ with
continuous marginal distribution functions (df's) $\mathbf{F}$ and
$\mathbf{G}$ respectively.$\ $Suppose that $\mathbf{X}$ is truncated to the
right by $\mathbf{Y},$ in the sense that $\mathbf{X}_{i}$ is only observed
when $\mathbf{X}_{i}\leq\mathbf{Y}_{i}.$ This model of randomly truncated data
commonly finds its applications in such areas like astronomy, economics,
medicine and insurance. In the actuarial world, for instance, it is usual that
the insurer claim data do not correspond to the underlying losses, because
they are truncated from above. Indeed, when facing large claims, the insurance
company specifies an upper limit to the amounts to be paid out. The excesses
over this fixed threshold are then covered by a reinsurance company. This kind
of reinsurance is called excess-loss reinsurance \citep[see,
e.g.,][]{EKM-97}. In life insurance, the upper limit, which may be random, is
called the cedent company retention level whereas in non-life insurance, it is
called the deductible. The usefulness of the statistical analysis under random
truncation is shown in \cite{Hrbst-99}, where the author applies truncated
model techniques to estimate loss reserves for IBNR (incurred but not
reported) claim amounts.\textbf{\ }For a recent discussion on randomly
right-truncated insurance claims, one refers to \cite{EO-2008}. Some examples
of truncated data from astronomy and economics can be found in \cite{W-85} and
for applications in the analysis of AIDS data, see \cite{WM-1989}. In
reliability, a real dataset, consisting in lifetimes of automobile brake pads
and already considered by \cite{Lawless} in page 69, was recently analyzed in
\cite{GS2014} as an application of randomly truncated heavy-tailed
models.\textbf{ }Since the focus is on datasets that contain extreme values,
then it would be natural to assume that both survival functions $\overline
{\mathbf{F}}:=1-\mathbf{F}$\textbf{\ }and $\overline{\mathbf{G}}%
:=1-\mathbf{G}$ are regularly varying at infinity with tail indices
$\gamma_{1}>0$ and $\gamma_{2}>0$ respectively. That is, we have, for any
$x>0,$%
\begin{equation}
\lim_{z\rightarrow\infty}\frac{\overline{\mathbf{F}}\left(  xz\right)
}{\overline{\mathbf{F}}\left(  z\right)  }=x^{-1/\gamma_{1}}\text{ and }%
\lim_{z\rightarrow\infty}\frac{\overline{\mathbf{G}}\left(  xz\right)
}{\overline{\mathbf{G}}\left(  z\right)  }=x^{-1/\gamma_{2}}. \label{RV-1}%
\end{equation}
This class of distributions, which includes models such as Pareto, Burr,
Fr\'{e}chet, L\'{e}vy-stable and log-gamma, takes a prominent role in extreme
value theory. Also known as heavy-tailed, Pareto-type or Pareto-like
distributions, they provide appropriate descriptions for large insurance
claims, log-returns, large price fluctuations,
etc...\ \citep[see, e.g.,][]{Res06}.\medskip

\noindent Let us denote $\left(  X_{i},Y_{i}\right)  ,$ $i=1,...,n$ to be the
observed data, as copies of a couple of rv's $\left(  X,Y\right)  ,$
corresponding to the truncated sample $\left(  \mathbf{X}_{i},\mathbf{Y}%
_{i}\right)  ,$ $i=1,...,N,$ where $n=n_{N}$ is a sequence of discrete rv's.
By of the law of large numbers, we have $n/N\overset{\mathbf{P}}{\rightarrow
}p:=\mathbf{P}\left(  \mathbf{X}\leq\mathbf{Y}\right)  $ as$\ n\overset
{\mathbf{P}}{\rightarrow}\infty.$ For convenience, throughout the paper, the
convergence in probability of $n$ and/or any of its subsequences is simply
denoted by $\rightarrow$ instead of $\overset{\mathbf{P}}{\rightarrow}.$ The
joint distribution of $X_{i}$ and $Y_{i}$ is%
\[%
\begin{array}
[c]{ll}%
H\left(  x,y\right)   & :=\mathbf{P}\left(  X\leq x,Y\leq y\right)  \medskip\\
& =\mathbf{P}\left(  \mathbf{X}\leq x,\mathbf{Y}\leq y\mid\mathbf{X}%
\leq\mathbf{Y}\right)  =p^{-1}%
{\displaystyle\int_{0}^{y}}
\mathbf{F}\left(  \min\left(  x,z\right)  \right)  d\mathbf{G}\left(
z\right)  .
\end{array}
\]
The marginal df's of the observed $X^{\prime}s$ and $Y^{\prime}s,$
respectively denoted by $F$ and $G,$ are equal to%
\[
F\left(  x\right)  :=p^{-1}\int_{0}^{x}\overline{\mathbf{G}}\left(  z\right)
d\mathbf{F}\left(  z\right)  \text{ and }G\left(  y\right)  :=p^{-1}\int
_{0}^{y}\mathbf{F}\left(  z\right)  d\mathbf{G}\left(  z\right)  .
\]
It follows that the corresponding tails are%
\begin{equation}
\overline{F}\left(  x\right)  =-p^{-1}\int_{x}^{\infty}\overline{\mathbf{G}%
}\left(  z\right)  d\overline{\mathbf{F}}\left(  z\right)  \text{ and
}\overline{G}\left(  y\right)  =-p^{-1}\int_{y}^{\infty}\mathbf{F}\left(
z\right)  d\overline{\mathbf{G}}\left(  z\right)  .\label{FbarGbar}%
\end{equation}
It is clear that the asymptotic behavior of $\overline{F}$ simultaneously
depends on $\overline{\mathbf{G}}$ and $\overline{\mathbf{F}},$ while that of
$\overline{G}$ only relies on $\overline{\mathbf{G}}\mathbf{.}$ Making use of
Proposition B.1.10 in \cite{deHF06},\ for the regularly varying functions
$\overline{\mathbf{F}}$ and $\overline{\mathbf{G}},$ we may readily show that
both $\overline{G}$ and $\overline{F}$ are regularly varying at infinity as
well, with respective tail indices $\gamma_{2}$ and $\gamma:=\gamma_{1}%
\gamma_{2}/\left(  \gamma_{1}+\gamma_{2}\right)  .$ It is worth noting that
the issue of analyzing extreme values in the context of random truncation, is
at an early stage. Indeed, the first contribution was made in the recent paper
of \cite{GS2015}, where the authors exploited the above relation between the
three indices to define an estimator of $\gamma_{1}$ by considering the
classical Hill estimators of $\gamma$ and $\gamma_{2}$ \citep{Hill75} as
functions of two distinct numbers of top statistics. However, they did not
handle the case where these numbers are equal because of the difficulty in
assessing the dependence between the two Hill estimators. In the present work,
we introduce a tail product-limit process for which we provide a weighted
Gaussian approximation as well. This tool will be very helpful when dealing
with the estimation of any tail related quantity.\textbf{\ }In particular, it
will lead to the asymptotic normality of the extreme value index estimator
that we define, under random right-truncation, as a function of a single
sample fraction of upper order statistics. But, prior to describing our
estimation methodology, let us note that, as mentioned by \cite{GS2015}, in
order to ensure that it remains enough extreme data for the inference to be
accurate, we need to impose the condition $\gamma_{1}<\gamma_{2}.$ In other
words, we consider the situation where the tail of the rv of interest
$\mathbf{X}$ is not too contaminated by the truncation rv $\mathbf{Y.}$ Since
$\mathbf{F}$ and $\mathbf{G}$ are heavy-tailed, then their right endpoints are
infinite and thus they are equal. Hence, from \cite{W-85}, we may write%
\begin{equation}
\int_{x}^{\infty}\frac{d\mathbf{F}\left(  y\right)  }{\mathbf{F}\left(
y\right)  }=\int_{x}^{\infty}\frac{dF\left(  y\right)  }{C\left(  y\right)
},\label{int}%
\end{equation}
where%
\begin{equation}
C\left(  z\right)  :=\mathbf{P}\left(  X\leq z\leq Y\right)  =F\left(
z\right)  -G\left(  z\right)  .\label{C}%
\end{equation}
Differentiating $\left(  \ref{int}\right)  $ leads to the following crucial
equation%
\begin{equation}
C\left(  x\right)  d\mathbf{F}\left(  x\right)  =\mathbf{F}\left(  x\right)
dF\left(  x\right)  ,\label{df}%
\end{equation}
\citep[see, for instance,][]{SS09}, whose solution is defined by
$\mathbf{F}\left(  x\right)  =\exp-\int_{x}^{\infty}dF\left(  z\right)
/C\left(  z\right)  .$\textbf{ }Replacing $F$ and $C$ by their respective
empirical counterparts%
\[
F_{n}\left(  x\right)  :=n^{-1}\sum\limits_{i=1}^{n}\mathbf{1}\left(
X_{i}\leq x\right)  \text{ and }C_{n}\left(  x\right)  :=n^{-1}\sum
\limits_{i=1}^{n}\mathbf{1}\left(  X_{i}\leq x\leq Y_{i}\right)  ,
\]
yields the product-limit estimator%
\[
\mathbf{F}_{n}\left(  x\right)  :=\prod_{i:X_{i}>x}\exp\left\{  -\frac
{1}{nC_{n}\left(  X_{i}\right)  }\right\}  ,
\]
to the underlying df $\mathbf{F.}$ The first mathematical investigation on
this estimator may be attributed to \cite{W-85}\ and the central limit theorem
under random truncation was established by \cite{sw2008}. Note that the
approximation $\exp\left(  -t\right)  \sim1-t,$ for small $t>0,$ results in
the well-known estimator introduced by \cite{LB-71}. Let us now introduce a
tail product-limit process corresponding to $\mathbf{F}_{n}\ $as follows:%
\begin{equation}
\mathbf{D}_{n}\left(  x\right)  :=\sqrt{k}\left(  \frac{\overline{\mathbf{F}%
}_{n}\left(  xX_{n-k:n}\right)  }{\overline{\mathbf{F}}_{n}\left(
X_{n-k:n}\right)  }-x^{-1/\gamma_{1}}\right)  ,\text{ }x>0,\label{KM-P}%
\end{equation}
where $X_{1:n}\leq...\leq X_{1:n}$ denote the order statistics pertaining to
$X_{1},...,X_{n}$ and $k=k_{n}$ is a sequence of discrete rv's satisfying%
\begin{equation}
1<k<n,\text{ }k\rightarrow\infty\text{ and }k/n\rightarrow0\text{ as
}n\rightarrow\infty.\label{K}%
\end{equation}
Observe that, in the case of complete data we have $\mathbf{F}_{n}\equiv
F_{n}$ with $\overline{F}_{n}\left(  X_{n-k:n}\right)  =k/n$ and thus the
process defined in $\left(  \ref{KM-P}\right)  $ becomes%
\[
D_{n}\left(  x\right)  :=\sqrt{k}\left(  \dfrac{n}{k}\overline{F}_{n}\left(
xX_{n-k:n}\right)  -x^{-1/\gamma_{1}}\right)  .
\]
By jointly applying Theorems 2.4.8 and 5.1.4 (pages 52 and 161) in
\cite{deHF06} we have that, for $x_{0}>0$ and $0<\xi<1/2,$%
\begin{equation}
\sup_{x\geq x_{0}}x^{(1/2-\xi)/\gamma_{1}}\left\vert D_{n}\left(  x\right)
-\Gamma\left(  x;W\right)  -x^{-1/\gamma_{1}}\dfrac{x^{\tau_{1}/\gamma_{1}}%
-1}{\tau_{1}\gamma_{1}}\sqrt{k}A_{0}\left(  n/k\right)  \right\vert
\overset{\mathbf{P}}{\rightarrow}0,\label{TP-C}%
\end{equation}
provided that $F$ fulfills the second-order regular variation condition with
auxiliary function $A_{0}$ tending to zero, not changing sign near infinity,
having a regularly varying absolute value with index $\tau_{1}<0$ and
satisfying $\sqrt{k}A_{0}\left(  n/k\right)  =O\left(  1\right)  .$ Here
$\Gamma\left(  x;W\right)  :=W\left(  x^{-1/\gamma_{1}}\right)  -x^{-1/\gamma
_{1}}W\left(  1\right)  $ with $\left\{  W\left(  s\right)  ;\text{ }0\leq
s\leq1\right\}  $ being a standard Wiener process. Many authors used this
approximation to establish the limit distributions of several statistics of
heavy-tailed distributions, such as tail index estimators (see, e.g.,
\citeauthor{deHF06}, \citeyear[page 76]{deHF06}) and goodness-of-fit
statistics \citep{KP-2008}. The main goal of this paper is to provide an
analogous result to $\left(  \ref{TP-C}\right)  $ in the random truncation
setting through the tail product-limit process $\left(  \ref{KM-P}\right)  ,$
which, to the best of our knowledge, was not addressed yet in the extreme
value theory literature.\medskip

\noindent The rest of the paper is organized as follows. In Section
\ref{sec2}, we present our main result which consists in a Gaussian
approximation to the tail product-limit process $\mathbf{D}_{n}\left(
x\right)  .$ As an application, we introduce, in Section \ref{sec3}, a new
Hill-type estimator \citep{Hill75} for the tail index $\gamma_{1}\ $and we
establish its consistency and asymptotic normality. The finite sample behavior
of the proposed estimator is checked by simulation in Section \ref{sec4}. The
proofs are postponed to Section \ref{sec5} and some results that are
instrumental to our needs are gathered in two lemmas in the Appendix.

\section{\textbf{Main results\label{sec2}}}

\noindent Weak approximations of extreme value theory based statistics are
achieved in the second-order framework \citep[see][]{deHS96}. Thus, it seems
quite natural to suppose that df's $\mathbf{F}$\ and $\mathbf{G}$\ satisfy the
well-known second-order condition of regular variation that we express in
terms of the tail quantile functions. That is, we assume that for $x>0,$ we
have%
\begin{equation}
\underset{t\rightarrow\infty}{\lim}\dfrac{\mathbb{U}_{\mathbf{F}}\left(
tx\right)  /\mathbb{U}_{\mathbf{F}}\left(  t\right)  -x^{\gamma_{1}}%
}{\mathbf{A}_{\mathbf{F}}\left(  t\right)  }=x^{\gamma_{1}}\dfrac{x^{\tau_{1}%
}-1}{\tau_{1}}, \label{second-order}%
\end{equation}
and%
\begin{equation}
\underset{t\rightarrow\infty}{\lim}\dfrac{\mathbb{U}_{\mathbf{G}}\left(
tx\right)  /\mathbb{U}_{\mathbf{G}}\left(  t\right)  -x^{\gamma_{2}}%
}{\mathbf{A}_{\mathbf{G}}\left(  t\right)  }=x^{\gamma_{2}}\dfrac{x^{\tau_{2}%
}-1}{\tau_{2}}, \label{second-orderG}%
\end{equation}
where $\tau_{1},\tau_{2}<0$\ are the second-order parameters and
$\mathbf{A}_{\mathbf{F}},$ $\mathbf{A}_{\mathbf{G}}$\ are functions tending to
zero and not changing signs near infinity with regularly varying absolute
values at infinity with indices $\tau_{1},$ $\tau_{2}$ respectively. For any
df $K,$ the function $\mathbb{U}_{K}\left(  t\right)  :=K^{\leftarrow}\left(
1-1/t\right)  ,$ $t>1,$ stands for the tail quantile function.

\begin{theorem}
\label{Theorem1}Assume that both second-order conditions $(\ref{second-order}%
)$\ and $(\ref{second-orderG})$\ hold with $\gamma_{1}<\gamma_{2}.$\ Let
$k=k_{n}$\ be a sequence satisfying $(\ref{K}),$\ then there exist a
function\textbf{\ }$\mathbf{A}_{0}\left(  t\right)  \sim\mathbf{A}%
_{\mathbf{F}}\left(  1/\overline{\mathbf{F}}\left(  \mathbb{U}_{F}\left(
t\right)  \right)  \right)  $\textbf{\ }and a standard Wiener process
$\left\{  \mathbf{W}\left(  s\right)  ;\text{ }0\leq s\leq1\right\}
,$\ defined on the probability space $\left(  \Omega,\mathcal{A}%
,\mathbf{P}\right)  ,$\ such that, for $0<\xi<1/2-\gamma/\gamma_{2}$ and
$x_{0}>0,$\ we have%
\[
\sup_{x\geq x_{0}}x^{\left(  1/2-\xi\right)  /\gamma-1/\gamma_{2}}\left\vert
\mathbf{D}_{n}\left(  x\right)  -\mathbf{\Gamma}\left(  x;\mathbf{W}\right)
-x^{-1/\gamma_{1}}\dfrac{x^{\tau_{1}/\gamma_{1}}-1}{\gamma_{1}\tau_{1}}%
\sqrt{k}\mathbf{A}_{0}\left(  n/k\right)  \right\vert \overset{\mathbf{P}%
}{\rightarrow}0,
\]
as $n\rightarrow\infty,$ provided that $\sqrt{k}\mathbf{A}_{0}\left(
n/k\right)  =O\left(  1\right)  ,$\textbf{\ }where$\mathbb{\ }\left\{
\Gamma\left(  x;\mathbf{W}\right)  ;\text{ }x>0\right\}  $\textbf{\ }is a
Gaussian process defined by%
\begin{align*}
&  \mathbf{\Gamma}\left(  x;\mathbf{W}\right)
\begin{tabular}
[c]{l}%
$:=$%
\end{tabular}
\frac{\gamma}{\gamma_{1}}x^{-1/\gamma_{1}}\left\{  x^{1/\gamma}\mathbf{W}%
\left(  x^{-1/\gamma}\right)  -\mathbf{W}\left(  1\right)  \right\} \\
&  \ \ \ \ \ \ +\frac{\gamma}{\gamma_{1}+\gamma_{2}}x^{-1/\gamma_{1}}\int
_{0}^{1}s^{-\gamma/\gamma_{2}-1}\left\{  x^{1/\gamma}\mathbf{W}\left(
x^{-1/\gamma}s\right)  -\mathbf{W}\left(  s\right)  \right\}  ds.
\end{align*}

\end{theorem}

\begin{remark}
\label{remak1}A very large value of $\gamma_{2}$\textbf{\ }yields a $\gamma
$-value that is very close to $\gamma_{1},$ meaning that the really observed
sample is almost the whole dataset. In other words, the complete data case
corresponds to the situation when $1/\gamma_{2}\equiv0,$\ in which case we
have $\gamma\equiv\gamma_{1}.$\ It follows that%
\[
\frac{\gamma}{\gamma_{1}+\gamma_{2}}\int_{0}^{1}s^{-\gamma/\gamma_{2}%
-1}\left\{  x^{1/\gamma}\mathbf{W}\left(  x^{-1/\gamma}s\right)
-\mathbf{W}\left(  s\right)  \right\}  ds\equiv0,
\]
and therefore\textbf{\ }$\Gamma\left(  x;\mathbf{W}\right)  =\mathbf{W}\left(
x^{-1/\gamma_{1}}\right)  -x^{-1/\gamma_{1}}\mathbf{W}\left(  1\right)
,$\textbf{\ }which agrees with the weak approximation $(\ref{TP-C}).$
\end{remark}

\section{\textbf{Tail index estimation\label{sec3}}}

\noindent We start the construction of our estimator by noting that from
Theorem 1.2.2 in \cite{deHF06}, the first-order condition $(\ref{RV-1})$ (for
$\overline{\mathbf{F}}$) implies that%
\[
\lim_{t\rightarrow\infty}\frac{1}{\overline{\mathbf{F}}\left(  t\right)  }%
\int_{t}^{\infty}x^{-1}\overline{\mathbf{F}}\left(  x\right)  dx=\gamma_{1},
\]
which, by an integration by parts, becomes%
\begin{equation}
\lim_{t\rightarrow\infty}\frac{1}{\overline{\mathbf{F}}\left(  t\right)  }%
\int_{t}^{\infty}\log\dfrac{x}{t}d\mathbf{F}(x)=\gamma_{1}. \label{log}%
\end{equation}
Replacing $\mathbf{F}$ by $\mathbf{F}_{n}$ and letting $t=X_{n-k:n}$ yields%
\[
\widehat{\gamma}_{1}:=\frac{1}{\overline{\mathbf{F}}_{n}\left(  X_{n-k:n}%
\right)  }\int_{X_{n-k:n}}^{\infty}\log\dfrac{x}{X_{n-k:n}}d\mathbf{F}%
_{n}\left(  x\right)  ,
\]
as a new estimator to $\gamma_{1}.$ By setting $\varphi_{n}^{\left(  1\right)
}\left(  x\right)  :=\mathbf{1}\left\{  x\geq X_{n-k:n}\right\}  \log\left(
x/X_{n-k:n}\right)  $ and $\varphi_{n}^{\left(  2\right)  }\left(  x\right)
:=\mathbf{1}\left\{  x\geq X_{n-k:n}\right\}  ,$ this may be rewritten into%
\[
\widehat{\gamma}_{1}=\frac{\int_{0}^{\infty}\varphi_{n}^{\left(  1\right)
}\left(  x\right)  d\mathbf{F}_{n}\left(  x\right)  }{\int_{0}^{\infty}%
\varphi_{n}^{\left(  2\right)  }\left(  x\right)  d\mathbf{F}_{n}\left(
x\right)  }.
\]
From the empirical counterpart of equation $(\ref{df})$ we get%
\[
\int_{0}^{\infty}\varphi_{n}^{\left(  1\right)  }\left(  x\right)
d\mathbf{F}_{n}\left(  x\right)  =\frac{1}{n}\sum_{i=n-k}^{n}\frac
{\mathbf{F}_{n}\left(  X_{i:n}\right)  }{C_{n}\left(  X_{i:n}\right)  }%
\log\left(  X_{i:n}/X_{n-k:n}\right)  ,
\]
and%
\[
\int_{0}^{\infty}\varphi_{n}^{\left(  2\right)  }\left(  x\right)
d\mathbf{F}_{n}\left(  x\right)  =\frac{1}{n}\sum_{i=n-k}^{n}\dfrac
{\mathbf{F}_{n}\left(  X_{i:n}\right)  }{C_{n}\left(  X_{i:n}\right)  }.
\]
Finally, changing $i$ to $n-i+1$ yields%
\[
\widehat{\gamma}_{1}=\left(  \sum\limits_{i=1}^{k}\frac{\mathbf{F}_{n}\left(
X_{n-i+1:n}\right)  }{C_{n}\left(  X_{n-i+1:n}\right)  }\right)  ^{-1}%
\sum_{i=1}^{k}\frac{\mathbf{F}_{n}\left(  X_{n-i+1:n}\right)  }{C_{n}\left(
X_{n-i+1:n}\right)  }\log\frac{X_{n-i+1:n}}{X_{n-k:n}}.
\]

\begin{remark}
\label{remak2}For complete data, we have $\mathbf{F}_{n}\mathbf{\equiv}%
F_{n}\mathbf{\equiv}C_{n}$ and consequently $\widehat{\gamma}_{1}$ reduces to
the classical Hill estimator \citep{Hill75}.
\end{remark}

\begin{theorem}
\label{Theorem2}Assume that $\left(  \ref{RV-1}\right)  $ holds with
$\gamma_{1}<\gamma_{2}$ and let $k=k_{n}$ be an integer sequence satisfying
$(\ref{K}).$ Then $\widehat{\gamma}_{1}\rightarrow\gamma_{1}$ in probability.
Assume further that both second-order conditions $(\ref{second-order})$\ and
$(\ref{second-orderG})$\ hold and $\sqrt{k}\mathbf{A}_{\mathbf{0}}\left(
n/k\right)  =O\left(  1\right)  ,$\ then%
\begin{align*}
\sqrt{k}\left(  \widehat{\gamma}_{1}-\gamma_{1}\right)   &  =\frac{\sqrt
{k}\mathbf{A}_{0}\left(  n/k\right)  }{1-\tau_{1}}-\gamma\mathbf{W}\left(
1\right) \\
&  +\frac{\gamma}{\gamma_{1}+\gamma_{2}}\int_{0}^{1}\left(  \gamma_{2}%
-\gamma_{1}-\gamma\log s\right)  s^{-\gamma/\gamma_{2}-1}\mathbf{W}\left(
s\right)  ds+o_{\mathbf{P}}\left(  1\right)  .
\end{align*}

\end{theorem}

\begin{corollary}
\label{Cor1}If, in addition to the assumptions of Theorem \ref{Theorem2}, we
suppose that $\sqrt{k}\mathbf{A}_{\mathbf{0}}\left(  n/k\right)
\rightarrow\lambda,$ then%
\[
\sqrt{k}\left(  \widehat{\gamma}_{1}-\gamma_{1}\right)  \overset{\mathcal{D}%
}{\rightarrow}\mathcal{N}\left(  \frac{\lambda}{1-\tau_{1}},\sigma^{2}\right)
,\text{ as }n\rightarrow\infty,
\]
where%
\[
\sigma^{2}:=\gamma^{2}\left(  1+\gamma_{1}/\gamma_{2}\right)  \left(
1+\left(  \gamma_{1}/\gamma_{2}\right)  ^{2}\right)  /\left(  1-\gamma
_{1}/\gamma_{2}\right)  ^{3}.
\]

\end{corollary}

\begin{remark}
\label{remak3}In the case of complete data we have, from Remark $\ref{remak1}%
,$ $\sigma^{2}\equiv\gamma_{1}^{2}.$\ It follows that $\sqrt{k}\left(
\widehat{\gamma}_{1}-\gamma_{1}\right)  \overset{\mathcal{D}}{\rightarrow
}\mathcal{N}\left(  \lambda/\left(  1-\tau_{1}\right)  ,\gamma_{1}^{2}\right)
,$\ as $n\rightarrow\infty,$\ which meets the asymptotic normality of the
classical Hill estimator \citep{Hill75}, see for instance, Theorem 3.2.5 in
\cite{deHF06}.
\end{remark}

\section{\textbf{Simulation study\label{sec4}}}

\noindent This study, just intended for illustrating the performance of our
estimator, is realized through two sets of truncated and truncation data, both
drawn from Burr's model:%
\[
\overline{\mathbf{F}}\left(  x\right)  =\left(  1+x^{1/\delta}\right)
^{-\delta/\gamma_{1}},\text{ }\overline{\mathbf{G}}\left(  x\right)  =\left(
1+x^{1/\delta}\right)  ^{-\delta/\gamma_{2}},\text{ }x\geq0,
\]
where $\delta,\gamma_{1},\gamma_{2}>0.$ The corresponding percentage of
observed data is equal to $p=\gamma_{2}/(\gamma_{1}+\gamma_{2}).$ We fix
$\delta=1/4$ and choose the values $0.6$ and $0.8$ for $\gamma_{1}$ and
$70\%,$ $80\%$ and $90\%$ for $p.$ For each couple $\left(  \gamma
_{1},p\right)  ,$ we solve the equation $p=\gamma_{2}/(\gamma_{1}+\gamma_{2})$
to get the pertaining $\gamma_{2}$-value. We vary the common size $N$ of both
samples $\left(  \mathbf{X}_{1},...,\mathbf{X}_{N}\right)  $ and $\left(
\mathbf{Y}_{1},...,\mathbf{Y}_{N}\right)  ,$ then for each size, we generate
$1000$ independent replicates. Our overall results are taken as the empirical
means of the results obtained through all repetitions. To determine the
optimal number (that we denote by $k^{\ast})$ of upper order statistics used
in the computation of $\widehat{\gamma}_{1},$ we apply the algorithm of
\cite[page 137]{ReTo7}. The performance of the newly defined estimator, in
terms of absolute bias and root of the mean squared error (rmse) is summarized
in Table \ref{Tab1}, where we see that, as expected, the size of the initial
sample influences the estimation: the larger $N,$ the better the estimation.
On the other hand, we note that the estimation accuracy decreases when the
truncation percentage increases, which seems logical. Finally, we observe that
the estimation of the larger value of the tail index is less precise.%

\begin{table}[h] \centering
$%
\begin{tabular}
[c]{c|cccccccc}\hline
& \multicolumn{8}{c}{${\small p=0.7}$}\\\hline
& \multicolumn{4}{c}{$\gamma_{1}=0.6$} & \multicolumn{4}{c}{$\gamma_{1}=0.8$%
}\\\hline
\multicolumn{1}{l|}{$N$} & $n$ & $k^{\ast}$ & {\small absolute bias} &
{\small rmse} & \multicolumn{1}{|c}{$n$} & $k^{\ast}$ & {\small absolute bias}
& {\small rmse}\\\hline\hline
\multicolumn{1}{l|}{${\small 200}$} & \multicolumn{1}{|r}{$139$} & $8$ &
$0.1811$ & \multicolumn{1}{r||}{$0.4645$} & \multicolumn{1}{r}{$140$} & $9$ &
$0.2485$ & $0.6034$\\
\multicolumn{1}{l|}{${\small 300}$} & \multicolumn{1}{|r}{$210$} & $17$ &
$0.1280$ & \multicolumn{1}{r||}{$0.3451$} & \multicolumn{1}{r}{$209$} & $23$ &
$0.1230$ & $0.5082$\\
\multicolumn{1}{l|}{${\small 500}$} & \multicolumn{1}{|r}{$348$} & $28$ &
$0.1151$ & \multicolumn{1}{r||}{$0.2803$} & \multicolumn{1}{r}{$348$} & $31$ &
$0.1024$ & $0.3708$\\\hline
\multicolumn{1}{l|}{${\small 1000}$} & \multicolumn{1}{|r}{$699$} & $43$ &
$0.0461$ & \multicolumn{1}{r||}{$0.2421$} & \multicolumn{1}{|r}{$700$} & $42$
& $0.0684$ & \multicolumn{1}{r}{$0.3825$}\\\hline
$1500$ & \multicolumn{1}{|r}{$1050$} & $67$ & $0.0212$ &
\multicolumn{1}{r||}{$0.2362$} & \multicolumn{1}{|c}{$1050$} & $57$ & $0.0527$
& $0.2539$\\\hline
$2000$ & \multicolumn{1}{|r}{$1399$} & $95$ & $0.0254$ &
\multicolumn{1}{r||}{$0.2261$} & \multicolumn{1}{|r}{$1401$} & $94$ & $0.0469$
& \multicolumn{1}{r}{$0.2602$}\\\hline
& \multicolumn{8}{c}{${\small p=0.8}$}\\\hline
\multicolumn{1}{l|}{${\small 200}$} & $159$ & $10$ & $0.0759$ & $0.5235$ &
\multicolumn{1}{||c}{$159$} & $11$ & $0.1435$ & $0.6815$\\\hline
\multicolumn{1}{l|}{${\small 300}$} & $240$ & $21$ & $0.0485$ & $0.3647$ &
\multicolumn{1}{||r}{$238$} & $18$ & $0.1052$ & $0.4511$\\\hline
\multicolumn{1}{l|}{${\small 500}$} & $399$ & $34$ & $0.0403$ & $0.2718$ &
\multicolumn{1}{||c}{$400$} & $41$ & $0.0935$ & $0.3170$\\\hline
${\small 1000}$ & $800$ & $55$ & $0.0408$ & $0.1999$ &
\multicolumn{1}{||c}{$800$} & $54$ & $0.0504$ & $0.2881$\\\hline
${\small 1500}$ & $1205$ & $116$ & $0.0562$ & $0.1911$ &
\multicolumn{1}{||c}{$1199$} & $85$ & $0.0428$ & $0.2362$\\\hline
${\small 2000}$ & $1599$ & $117$ & $0.0285$ & $0.1534$ &
\multicolumn{1}{||c}{$1599$} & $117$ & $0.0373$ & $0.2054$\\\hline
& \multicolumn{8}{c}{${\small p=0.9}$}\\\hline
\multicolumn{1}{l|}{${\small 200}$} & \multicolumn{1}{|r}{$180$} & $14$ &
$0.0537$ & $0.5204$ & \multicolumn{1}{||c}{$179$} & $17$ & $0.1098$ &
$0.7531$\\\hline
\multicolumn{1}{l|}{${\small 300}$} & \multicolumn{1}{|r}{$269$} & $24$ &
$0.0388$ & $0.3343$ & \multicolumn{1}{||r}{$268$} & $26$ & $0.0844$ &
$0.3654$\\\hline
\multicolumn{1}{l|}{${\small 500}$} & \multicolumn{1}{|r}{$450$} & $48$ &
$0.0294$ & $0.2869$ & \multicolumn{1}{||r}{$450$} & $49$ & $0.0721$ &
$0.2448$\\\hline
${\small 1000}$ & \multicolumn{1}{|r}{$899$} & $64$ & $0.0359$ & $0.1557$ &
\multicolumn{1}{||r}{$899$} & $67$ & $0.0465$ & $0.2014$\\\hline
${\small 1500}$ & \multicolumn{1}{|r}{$1349$} & $103$ & $0.0171$ & $0.1350$ &
\multicolumn{1}{||r}{$1349$} & $101$ & $0.0385$ & $0.1828$\\\hline
${\small 2000}$ & \multicolumn{1}{|r}{$1799$} & $144$ & $0.0188$ & $0.1107$ &
\multicolumn{1}{||r}{$1799$} & $145$ & $0.0251$ & $0.1466$\\\hline
\end{tabular}
\ $\medskip
\caption{Absolute bias and rmse of the tail index estimator based on 1000 right-truncated samples of Burr models.}\label{Tab1}%
\end{table}%

\section{\textbf{Proofs\label{sec5}}}

\subsection{Proof of Theorem \ref{Theorem1}}

\noindent Set $U_{i}:=\overline{F}\left(  X_{i}\right)  $ and define the
corresponding uniform tail empirical process by $\alpha_{n}\left(  s\right)
:=\sqrt{k}\left(  \mathbf{U}_{n}\left(  s\right)  -s\right)  ,$ for $0\leq
s\leq1,$ where $\mathbf{U}_{n}\left(  s\right)  :=k^{-1}\sum_{i=1}%
^{n}\mathbf{1}\left(  U_{i}<ks/n\right)  .$ The weighted weak approximation to
$\alpha_{n}\left(  s\right)  $ given in terms of, either a sequence of Wiener
processes (see, e.g., \citeauthor{E-1992}, \citeyear{E-1992} and
\citeauthor{DDL-2006}, \citeyear{DDL-2006})\textbf{\ }or a single Wiener
process as in Proposition 3.1 of \cite{EHL-2006}, will be very crucial to our
proof procedure. In the sequel, we use the latter representation which says
that: there exists a Wiener process $\mathbf{W},$ such that for every
$0\leq\eta<1/2,$%
\begin{equation}
\sup_{0<s\leq1}s^{-\eta}\left\vert \alpha_{n}\left(  s\right)  -\mathbf{W}%
\left(  s\right)  \right\vert \overset{\mathbf{p}}{\rightarrow}0,\text{ as
}n\rightarrow\infty. \label{approx}%
\end{equation}
We begin by fixing $x_{0}>0,$ then we decompose $k^{-1/2}\mathbf{D}_{n}\left(
x\right)  ,$ for $x\geq x_{0},$ as the sum of the following four terms:%
\[
\mathbf{M}_{n1}\left(  x\right)  :=x^{-1/\gamma_{1}}\frac{\overline
{\mathbf{F}}_{n}\left(  xX_{n-k:n}\right)  -\overline{\mathbf{F}}\left(
xX_{n-k:n}\right)  }{\overline{\mathbf{F}}\left(  xX_{n-k:n}\right)
},\medskip
\]%
\[
\mathbf{M}_{n2}\left(  x\right)  :=-\frac{\overline{\mathbf{F}}\left(
xX_{n-k:n}\right)  }{\overline{\mathbf{F}}_{n}\left(  X_{n-k:n}\right)  }%
\frac{\overline{\mathbf{F}}_{n}\left(  X_{n-k:n}\right)  -\overline
{\mathbf{F}}\left(  X_{n-k:n}\right)  }{\overline{\mathbf{F}}\left(
X_{n-k:n}\right)  },\medskip
\]%
\[
\mathbf{M}_{n3}\left(  x\right)  :=\left(  \frac{\overline{\mathbf{F}}\left(
xX_{n-k:n}\right)  }{\overline{\mathbf{F}}_{n}\left(  X_{n-k:n}\right)
}-x^{-1/\gamma_{1}}\right)  \frac{\overline{\mathbf{F}}_{n}\left(
xX_{n-k:n}\right)  -\overline{\mathbf{F}}\left(  xX_{n-k:n}\right)
}{\overline{\mathbf{F}}\left(  xX_{n-k:n}\right)  }\medskip
\]
and%
\[
\mathbf{M}_{n4}\left(  x\right)  :=\frac{\overline{\mathbf{F}}\left(
xX_{n-k:n}\right)  }{\overline{\mathbf{F}}\left(  X_{n-k:n}\right)
}-x^{-1/\gamma_{1}}.
\]
In order to establish the result of the theorem, we will successively show
that, under the first-order of regular variation conditions, we have uniformly
on $x\geq x_{0},$ for $\gamma/\gamma_{2}<\eta<1/2$ and $\epsilon>0$
sufficiently small%
\begin{align*}
&  x^{1/\gamma_{1}}\sqrt{k}\mathbf{M}_{n1}\left(  x\right) \\
&  =x^{1/\gamma}\left\{  \frac{\gamma}{\gamma_{1}}\mathbf{W}\left(
x^{-1/\gamma}\right)  +\frac{\gamma}{\gamma_{1}+\gamma_{2}}\int_{0}%
^{1}t^{-\gamma/\gamma_{2}-1}\mathbf{W}\left(  x^{-1/\gamma}t\right)
dt\right\}  +O_{\mathbf{p}}\left(  \epsilon\right)  x^{\left(  1-\eta\right)
/\gamma\pm\epsilon},\medskip
\end{align*}%
\[
x^{1/\gamma_{1}}\sqrt{k}\mathbf{M}_{n2}\left(  x\right)  =-\left\{
\frac{\gamma}{\gamma_{1}}\mathbf{W}\left(  1\right)  +\frac{\gamma}{\gamma
_{1}+\gamma_{2}}\int_{0}^{1}t^{-\gamma/\gamma_{2}-1}\mathbf{W}\left(
t\right)  dt\right\}  +O_{\mathbf{p}}\left(  \epsilon\right)  x^{\pm\epsilon
},\medskip
\]
and%
\[
x^{1/\gamma_{1}}\sqrt{k}\mathbf{M}_{n3}\left(  x\right)  =O_{\mathbf{p}%
}\left(  \epsilon\right)  x^{-1/\gamma_{1}+\left(  1-\eta\right)  /\gamma
\pm\epsilon}.
\]
Moreover, if we assume the second-order condition we will show that%
\[
x^{1/\gamma_{1}}\sqrt{k}\mathbf{M}_{n4}\left(  x\right)  =\left(
1+o_{\mathbf{p}}\left(  1\right)  \right)  \dfrac{x^{\tau_{1}/\gamma_{1}}%
-1}{\gamma_{1}\tau_{1}}\sqrt{k}\mathbf{A}_{0}\left(  n/k\right)  .
\]
Here $O_{\mathbf{p}}$ and $o_{\mathbf{p}}$ stand for the usual stochastic
order symbols. For convenience, let $a_{k}:=\mathbb{U}_{F}\left(  n/k\right)
$ and recall that $\mathbb{U}_{F}$ is regularly varying (with index $\gamma).$
Then by combining Corollary 2.2.2 with Proposition B.1.10 in \cite{deHF06}, we
show that $X_{n-k:n}/a_{k}\overset{\mathbf{P}}{\rightarrow}1$ as
$n\rightarrow\infty,$ which implies, due to the regular variation of
$\overline{\mathbf{F}},$ that $\overline{\mathbf{F}}\left(  xa_{k}\right)
/\overline{\mathbf{F}}\left(  xX_{n-k:n}\right)  $ $=1+o_{\mathbf{p}}\left(
1\right)  $ and therefore%
\begin{equation}
\mathbf{M}_{n1}\left(  x\right)  =\left(  1+o_{\mathbf{p}}\left(  1\right)
\right)  \mathbf{M}_{n1}^{\ast}\left(  x\right)  , \label{Mn1}%
\end{equation}
where%
\[
\mathbf{M}_{n1}^{\ast}\left(  x\right)  :=x^{-1/\gamma_{1}}\frac
{\overline{\mathbf{F}}_{n}\left(  xX_{n-k:n}\right)  -\overline{\mathbf{F}%
}\left(  xX_{n-k:n}\right)  }{\overline{\mathbf{F}}\left(  xa_{k}\right)  }.
\]
Now, observe that, in view of equation $(\ref{df}),$ we may write%
\[
\mathbf{F}\left(  x\right)  =\exp-\Lambda\left(  x\right)  \text{ and
}\mathbf{F}_{n}\left(  x\right)  =\exp-\Lambda_{n}\left(  x\right)  ,
\]
where $\Lambda\left(  x\right)  $ and its empirical counterpart $\Lambda
_{n}\left(  x\right)  $ are defined by $\int_{x}^{\infty}dF\left(  z\right)
/C\left(  z\right)  $ and $\int_{x}^{\infty}dF_{n}\left(  z\right)
/C_{n}\left(  z\right)  $ respectively. Note that $\overline{\mathbf{F}}%
_{n}\left(  xX_{n-k:n}\right)  ,$ $\overline{\mathbf{F}}\left(  xX_{n-k:n}%
\right)  $ and $\overline{\mathbf{F}}\left(  xa_{k}\right)  $ tend to zero in
probability, uniformly on $x\geq x_{0},$ it follows that $\Lambda_{n}\left(
xX_{n-k:n}\right)  ,$ $\Lambda\left(  xX_{n-k:n}\right)  $ and $\Lambda\left(
xa_{k}\right)  $ go to zero in probability as well. Using the approximation
$1-\exp(-t)\sim t,$ as $t\downarrow0,$ we may write%
\[
x^{1/\gamma_{1}}\mathbf{M}_{n1}^{\ast}\left(  x\right)  =\left(
1+o_{\mathbf{p}}\left(  1\right)  \right)  \frac{\Lambda_{n}\left(
xX_{n-k:n}\right)  -\Lambda\left(  xX_{n-k:n}\right)  }{\Lambda\left(
xa_{k}\right)  }.
\]
Next, we provide a Gaussian approximation to the expression%
\[
\sqrt{k}\dfrac{\Lambda_{n}\left(  xX_{n-k:n}\right)  -\Lambda\left(
xX_{n-k:n}\right)  }{\Lambda\left(  xa_{k}\right)  },
\]
then we deduce one to $\sqrt{k}x^{1/\gamma_{1}}\mathbf{M}_{n1}^{\ast}\left(
x\right)  .$ For this, we decompose the difference $\Lambda_{n}\left(
xX_{n-k:n}\right)  -\Lambda\left(  xX_{n-k:n}\right)  $ in the sum of%
\[
S_{n1}\left(  x\right)  :=-\int_{xa_{k}}^{\infty}\frac{d\left(  \overline
{F}_{n}\left(  z\right)  -\overline{F}\left(  z\right)  \right)  }{C\left(
z\right)  },
\]%
\[
S_{n2}\left(  x\right)  :=-\int_{xX_{n-k:n}}^{\infty}\left\{  \frac{1}%
{C_{n}\left(  z\right)  }-\frac{1}{C\left(  z\right)  }\right\}  d\overline
{F}_{n}\left(  z\right)  ,
\]
and%
\[
S_{n3}\left(  x\right)  :=-\int_{xX_{n-k:n}}^{xa_{k}}\frac{d\left(
\overline{F}_{n}\left(  z\right)  -\overline{F}\left(  z\right)  \right)
}{C\left(  z\right)  }.
\]
For the first term, we use the fact that $\overline{F}_{n}\left(  z\right)
=0$ for $z\geq X_{n:n},$ to write, after an integration by parts and a change
of variables, $S_{n1}\left(  x\right)  =S_{n1}^{\left(  1\right)  }\left(
x\right)  -S_{n1}^{\left(  2\right)  }\left(  x\right)  ,$ with%
\[
S_{n1}^{\left(  1\right)  }\left(  x\right)  :=\frac{\overline{F}_{n}\left(
a_{k}x\right)  -\overline{F}\left(  a_{k}x\right)  }{C\left(  a_{k}x\right)
}\text{ and }S_{n1}^{\left(  2\right)  }\left(  x\right)  :=\int_{x}^{\infty
}\frac{\overline{F}_{n}\left(  a_{k}z\right)  -\overline{F}\left(
a_{k}z\right)  }{C^{2}\left(  a_{k}z\right)  }dC\left(  a_{k}z\right)  .
\]
It is easy to verify that $\overline{F}_{n}\left(  xa_{k}\right)
-\overline{F}\left(  xa_{k}\right)  =\dfrac{\sqrt{k}}{n}\alpha_{n}\left(
\dfrac{n}{k}\overline{F}\left(  xa_{k}\right)  \right)  ,$ it follows that%
\[
\frac{\sqrt{k}S_{n1}^{\left(  1\right)  }\left(  x\right)  }{\Lambda\left(
a_{k}x\right)  }=d_{n}\left(  x\right)  \alpha_{n}\left(  \frac{n}{k}%
\overline{F}\left(  a_{k}x\right)  \right)  ,
\]
where $d_{n}\left(  x\right)  :=\dfrac{k/n}{\Lambda\left(  xa_{k}\right)
C\left(  a_{k}x\right)  }.$ From Lemma $\ref{Lemma1}$ $\left(  iii\right)
,$\ we have%
\begin{equation}
d_{n}\left(  x\right)  =\left(  \gamma/\gamma_{1}\right)  x^{1/\gamma
}+O\left(  \epsilon\right)  x^{1/\gamma\pm\epsilon}, \label{dn}%
\end{equation}
as $n\rightarrow\infty,$ uniformly on $x\geq x_{0},$ it follows that%
\[
\frac{\sqrt{k}S_{n1}^{\left(  1\right)  }\left(  x\right)  }{\Lambda\left(
a_{k}x\right)  }=\left\{  \left(  \gamma/\gamma_{1}\right)  x^{1/\gamma
}+O_{\mathbf{p}}\left(  \epsilon\right)  x^{1/\gamma\pm\epsilon}\right\}
\alpha_{n}\left(  \frac{n}{k}\overline{F}\left(  a_{k}x\right)  \right)  .
\]
On the other hand, for $0<\eta<1/2,$ the sequence of rv's $\sup_{0<s\leq
1}\left\vert \alpha_{n}\left(  s\right)  \right\vert /s^{\eta}$ is
stochastically bounded. This comes from the inequality%
\[
\sup_{0<s\leq1}s^{-\eta}\left\vert \alpha_{n}\left(  s\right)  \right\vert
\leq\sup_{0<s\leq1}s^{-\eta}\left\vert \alpha_{n}\left(  s\right)
-\mathbf{W}\left(  s\right)  \right\vert +\sup_{0<s\leq1}s^{-\eta}\left\vert
\mathbf{W}\left(  s\right)  \right\vert ,
\]
with approximation $\left(  \ref{approx}\right)  $ and the fact $\sup
_{0<s\leq1}s^{-\eta}\left\vert \mathbf{W}\left(  s\right)  \right\vert
=O_{\mathbf{p}}\left(  1\right)  $
\citep[see, e.g., Lemma 3.2 in][]{EHL-2006}. Now, let $\epsilon>0$ be
sufficiently small. Then, by applying Potter's inequalities to $\overline{F}$
\citep[see, e.g., Proposition B.1.9, assertion 5  in][]{deHF06}, we write
$\dfrac{n}{k}\overline{F}\left(  a_{k}x\right)  \leq\left(  1+\epsilon\right)
x^{-1/\gamma\pm\epsilon},$ it follows that $\alpha_{n}\left(  \dfrac{n}%
{k}\overline{F}\left(  a_{k}x\right)  \right)  =O_{\mathbf{p}}\left(
x^{-\eta/\gamma\pm\eta\epsilon}\right)  .$ For notational simplicity and
without loss of generality, we attribute $\epsilon$ to any constant times
$\epsilon$ and $v^{\pm\epsilon}$ to any linear combinations of $v^{\pm
c_{1}\epsilon}$ and $v^{\pm c_{2}\epsilon},\ $for every $c_{1},c_{2}>0.$
Therefore%
\[
\frac{\sqrt{k}S_{n1}^{\left(  1\right)  }\left(  x\right)  }{\Lambda\left(
a_{k}x\right)  }=\frac{\gamma}{\gamma_{1}}x^{1/\gamma}\alpha_{n}\left(
\frac{n}{k}\overline{F}\left(  a_{k}x\right)  \right)  +O_{\mathbf{p}}\left(
\epsilon\right)  x^{\left(  1-\eta\right)  /\gamma\pm\epsilon}.
\]
For $S_{n1}^{\left(  2\right)  }\left(  x\right)  ,$ let us write%
\[
\frac{\sqrt{k}S_{n1}^{\left(  2\right)  }\left(  x\right)  }{\Lambda\left(
a_{k}x\right)  }=d_{n}\left(  x\right)  \frac{C\left(  a_{k}x\right)
}{C\left(  a_{k}\right)  }\int_{x}^{\infty}\dfrac{C^{2}\left(  a_{k}\right)
}{C^{2}\left(  a_{k}z\right)  }\alpha_{n}\left(  \dfrac{n}{k}\overline
{F}\left(  a_{k}z\right)  \right)  d\dfrac{C\left(  a_{k}z\right)  }{C\left(
a_{k}\right)  }.
\]
From Lemma $\ref{Lemma1}$ $\left(  i\right)  ,$ the function $C$ is regularly
varying at infinity with index $\left(  -1/\gamma_{2}\right)  ,$ as
$\overline{G}$ is, this implies that $C\left(  xa_{k}\right)  /C\left(
a_{k}\right)  =x^{-1/\gamma_{2}}+O\left(  \epsilon\right)  x^{-1/\gamma_{2}%
\pm\epsilon}.$ Then by using $\left(  \ref{dn}\right)  ,$ we get%
\begin{equation}
d_{n}\left(  x\right)  \frac{C\left(  a_{k}x\right)  }{C\left(  a_{k}\right)
}=\left(  \gamma/\gamma_{1}\right)  x^{1/\gamma_{1}}+O\left(  \epsilon\right)
x^{1/\gamma_{1}\pm\epsilon},\text{ as }n\rightarrow\infty. \label{dn2}%
\end{equation}
For convenience, we set $\sqrt{k}S_{n1}^{\left(  2\right)  }\left(  x\right)
/\Lambda\left(  a_{k}x\right)  =\left\{  1+O\left(  \epsilon\right)
x^{\pm\epsilon}\right\}  \mathcal{T}_{n}\left(  x\right)  ,$ where%
\[
\mathcal{T}_{n}\left(  x\right)  :=\frac{\gamma}{\gamma_{1}}x^{1/\gamma_{1}%
}\int_{x}^{\infty}\dfrac{C^{2}\left(  a_{k}\right)  }{C^{2}\left(
a_{k}z\right)  }\alpha_{n}\left(  \dfrac{n}{k}\overline{F}\left(
a_{k}z\right)  \right)  d\dfrac{C\left(  a_{k}z\right)  }{C\left(
a_{k}\right)  },
\]
which we decompose in the sum of%
\[
I_{n}\left(  x\right)  :=\frac{\gamma}{\gamma_{1}}x^{1/\gamma_{1}}\int
_{x}^{\infty}\dfrac{C^{2}\left(  a_{k}\right)  }{C^{2}\left(  a_{k}z\right)
}\alpha_{n}\left(  \dfrac{n}{k}\overline{F}\left(  a_{k}z\right)  \right)
d\dfrac{F\left(  a_{k}z\right)  }{C\left(  a_{k}\right)  },
\]%
\[
J_{n}\left(  x\right)  :=-\frac{\gamma}{\gamma_{1}}x^{1/\gamma_{1}}\int
_{x}^{\infty}\left\{  \dfrac{C^{2}\left(  a_{k}\right)  }{C^{2}\left(
a_{k}z\right)  }-z^{2/\gamma_{2}}\right\}  \alpha_{n}\left(  \dfrac{n}%
{k}\overline{F}\left(  a_{k}z\right)  \right)  d\dfrac{G\left(  a_{k}z\right)
}{C\left(  a_{k}\right)  },
\]
and%
\[
K_{n}\left(  x\right)  :=\frac{\gamma}{\gamma_{1}}x^{1/\gamma_{1}}\int
_{x}^{\infty}z^{2/\gamma_{2}}\alpha_{n}\left(  \dfrac{n}{k}\overline{F}\left(
a_{k}z\right)  \right)  d\dfrac{\overline{G}\left(  a_{k}z\right)  }{C\left(
a_{k}\right)  }.
\]
Recall that $a_{k}\rightarrow\infty,$ $C\left(  a_{k}\right)  \sim\overline
{G}\left(  a_{k}\right)  $ and $\overline{F}\left(  a_{k}\right)  =o\left(
\overline{G}\left(  a_{k}\right)  \right)  $ as $n\rightarrow\infty.$ On the
other hand, by using, once again, Potter's inequalities to $C,$ (regularly
varying at infinity with index $-1/\gamma_{2}),$ we write, for all large $n$
and $z\geq x,$%
\begin{equation}
\left(  1-\epsilon\right)  z^{-1/\gamma_{2}}\min\left(  z^{\epsilon
},z^{-\epsilon}\right)  \leq\dfrac{C\left(  a_{k}z\right)  }{C\left(
a_{k}\right)  }\leq\left(  1+\epsilon\right)  z^{-1/\gamma_{2}}\max\left(
z^{\epsilon},z^{-\epsilon}\right)  . \label{C_inequa1}%
\end{equation}
It is clear this implies that $C^{2}\left(  a_{k}\right)  /C^{2}\left(
a_{k}z\right)  \leq\left(  1-\epsilon\right)  ^{-2}z^{2/\gamma_{2}\pm
2\epsilon}.$ In view of the stochastic boundedness of $\sup_{0<s\leq
1}\left\vert \alpha_{n}\left(  s\right)  \right\vert /s^{\eta}$ and the fact
that $\dfrac{n}{k}\overline{F}\left(  a_{k}z\right)  \leq\left(
1+\epsilon\right)  z^{-1/\gamma\pm\epsilon},$ we have%
\[
I_{n}\left(  x\right)  =o_{\mathbf{p}}\left(  1\right)  x^{1/\gamma_{1}}%
\int_{x}^{\infty}z^{2/\gamma_{2}\mp2\epsilon}\left(  z^{-1/\gamma\pm\epsilon
}\right)  ^{\eta}d\dfrac{\overline{F}\left(  a_{k}z\right)  }{\overline
{F}\left(  a_{k}\right)  }.
\]
Integrating by parts, we readily get $I_{n}\left(  x\right)  =o_{\mathbf{p}%
}\left(  1\right)  x^{1/\gamma_{2}-\eta/\gamma\pm\epsilon}=o_{\mathbf{p}%
}\left(  1\right)  x^{\left(  1-\eta\right)  /\gamma\pm\epsilon}.$ Let us now
consider $J_{n}\left(  x\right)  .$ From Proposition B.1.10 in \cite{deHF06},
we have $\left\vert C\left(  a_{k}z\right)  /C\left(  a_{k}\right)
-z^{-1/\gamma_{2}}\right\vert \leq\epsilon z^{-1/\gamma_{2}\pm\epsilon}.$
Applying the mean value theorem, then combining this inequality with $\left(
\ref{C_inequa1}\right)  ,$ yield%
\[
\left\vert \dfrac{C^{2}\left(  a_{k}\right)  }{C^{2}\left(  a_{k}z\right)
}-z^{2/\gamma_{2}}\right\vert \leq\epsilon\frac{2\left(  z^{\pm\epsilon
}+1\right)  }{\left(  1-\epsilon\right)  ^{3}}z^{2/\gamma_{2}\pm\epsilon}.
\]
Similar arguments as the above lead to $J_{n}\left(  x\right)  =O_{\mathbf{p}%
}\left(  \epsilon\right)  x^{\left(  1-\eta\right)  /\gamma\pm\epsilon}.$ Now,
we focus on $K_{n}\left(  x\right)  .$ Since $C\left(  a_{k}\right)
\sim\overline{G}\left(  a_{k}\right)  ,$ then%
\[
K_{n}\left(  x\right)  =\left(  1+o_{\mathbf{p}}\left(  1\right)  \right)
\frac{\gamma}{\gamma_{1}}x^{1/\gamma_{1}}\int_{x}^{\infty}z^{2/\gamma_{2}%
}\alpha_{n}\left(  \dfrac{n}{k}\overline{F}\left(  a_{k}z\right)  \right)
d\dfrac{\overline{G}\left(  a_{k}z\right)  }{\overline{G}\left(  a_{k}\right)
}.
\]
Let $G^{\leftarrow}$ denote the quantile function pertaining to df $G$ and use
the change of variables $z=G^{\leftarrow}\left(  1-s\overline{G}\left(
a_{k}\right)  \right)  /a_{k}$ to get%
\[
K_{n}\left(  x\right)  =-\left(  1+o_{\mathbf{p}}\left(  1\right)  \right)
\frac{\gamma}{\gamma_{1}}x^{1/\gamma_{1}}\int_{0}^{\tfrac{\overline{G}\left(
a_{k}x\right)  }{\overline{G}\left(  a_{k}\right)  }}\left(  \frac
{G^{\leftarrow}\left(  1-s\overline{G}\left(  a_{k}\right)  \right)  }{a_{k}%
}\right)  ^{2/\gamma_{2}}\alpha_{n}\left(  \ell_{n}\left(  s\right)  \right)
ds,
\]
where $\ell_{n}\left(  s\right)  :=\dfrac{n}{k}\overline{F}\left(
G^{\leftarrow}\left(  1-s\overline{G}\left(  a_{k}\right)  \right)  \right)
.$ It is easy to check that
\[
K_{n}\left(  x\right)  =-\left(  1+o_{\mathbf{p}}\left(  1\right)  \right)
\sum_{i=1}^{3}K_{ni}\left(  x\right)  ,
\]
where%
\[
K_{n1}\left(  x\right)  :=\frac{\gamma}{\gamma_{1}}x^{\frac{1}{\gamma_{1}}%
}\int_{0}^{\tfrac{\overline{G}\left(  a_{k}x\right)  }{\overline{G}\left(
a_{k}\right)  }}\left\{  \left(  \frac{G^{\leftarrow}\left(  1-s\overline
{G}\left(  a_{k}\right)  \right)  }{a_{k}}\right)  ^{2/\gamma_{2}}%
-s^{-2}\right\}  \alpha_{n}\left(  \ell_{n}\left(  s\right)  \right)  ds,
\]%
\[
K_{n2}\left(  x\right)  :=\frac{\gamma}{\gamma_{1}}x^{1/\gamma_{1}}%
\int_{x^{-1/\gamma_{2}}}^{\tfrac{\overline{G}\left(  a_{k}x\right)
}{\overline{G}\left(  a_{k}\right)  }}s^{-2}\alpha_{n}\left(  \ell_{n}\left(
s\right)  \right)  ds,
\]
and%
\[
K_{n3}\left(  x\right)  :=\frac{\gamma}{\gamma_{1}}x^{1/\gamma_{1}}\int
_{0}^{x^{-1/\gamma_{2}}}s^{-2}\alpha_{n}\left(  \ell_{n}\left(  s\right)
\right)  ds.
\]
By routine manipulations and similar arguments based on stochastic boundedness
of $\sup_{0<s\leq1}\left\vert \alpha_{n}\left(  s\right)  \right\vert
/s^{\eta}$ and the aforementioned Proposition B.1.10 applied to the regularly
varying functions $\overline{G}$ and $G^{\leftarrow}\left(  1-\cdot\right)  ,$
we show that $K_{ni}\left(  x\right)  =O_{\mathbf{p}}\left(  \epsilon\right)
x^{\left(  1-\eta\right)  /\gamma\mp\epsilon},$ $i=1,2$ and $K_{n3}\left(
x\right)  =O_{\mathbf{p}}\left(  1\right)  x^{\left(  1-\eta\right)
/\gamma\mp\epsilon},$ therefore we omit the details. Up to this stage, we have
shown that $\mathcal{T}_{n}\left(  x\right)  =O_{\mathbf{p}}\left(  1\right)
x^{\left(  1-\eta\right)  /\gamma\mp\epsilon}.$\ It follows that
\[
\frac{\sqrt{k}S_{n1}^{\left(  2\right)  }\left(  x\right)  }{\Lambda\left(
a_{k}x\right)  }=\mathcal{T}_{n}\left(  x\right)  +O_{\mathbf{p}}\left(
\epsilon\right)  x^{\left(  1-\eta\right)  /\gamma\mp\epsilon},
\]
which, after gathering the components of $\mathcal{T}_{n}\left(  x\right)  ,$
is equal to%
\[
\frac{\gamma}{\gamma_{1}}x^{1/\gamma_{1}}\int_{0}^{x^{-1/\gamma_{2}}}%
s^{-2}\alpha_{n}\left(  \ell_{n}\left(  s\right)  \right)  ds+O_{\mathbf{p}%
}\left(  \epsilon\right)  x^{\left(  1-\eta\right)  /\gamma\mp\epsilon}.
\]
Therefore%
\begin{align*}
\frac{\sqrt{k}S_{n1}\left(  x\right)  }{\Lambda\left(  a_{k}x\right)  }  &
=\frac{\gamma}{\gamma_{1}}x^{1/\gamma}\alpha_{n}\left(  \frac{n}{k}%
\overline{F}\left(  a_{k}x\right)  \right) \\
&  -\frac{\gamma}{\gamma_{1}}x^{1/\gamma_{1}}\int_{0}^{x^{-1/\gamma_{2}}%
}s^{-2}\alpha_{n}\left(  \ell_{n}\left(  s\right)  \right)  ds+O_{\mathbf{p}%
}\left(  \epsilon\right)  x^{\left(  1-\eta\right)  /\gamma\mp\epsilon}.
\end{align*}
Recall that $\gamma_{1}<\gamma_{2}$ and $\gamma/\gamma_{2}=\gamma_{1}/\left(
\gamma_{1}+\gamma_{2}\right)  ,$ then we may choose the constant $\eta$ in
such a way that $\gamma/\gamma_{2}<\eta<1/2.$ Making use of weak
approximation\ $(\ref{approx}),$ we obtain%
\begin{align*}
&  \frac{\sqrt{k}S_{n1}\left(  x\right)  }{\Lambda\left(  a_{k}x\right)  }\\
&  =\frac{\gamma}{\gamma_{1}}x^{1/\gamma}\mathbf{W}\left(  \frac{n}%
{k}\overline{F}\left(  a_{k}x\right)  \right)  +\frac{\gamma}{\gamma_{1}%
}x^{1/\gamma}%
{\displaystyle\int_{0}^{x^{-1/\gamma_{2}}}}
s^{-2}\mathbf{W}\left(  \ell_{n}\left(  s\right)  \right)  ds+O_{\mathbf{p}%
}\left(  \epsilon\right)  x^{\left(  1-\eta\right)  /\gamma\mp\epsilon}.
\end{align*}
Note that $k/n=\overline{F}\left(  G^{\leftarrow}\left(  1-\overline{G}\left(
a_{k}\right)  \right)  \right)  ,$ hence%
\[
\ell_{n}\left(  s\right)  =\frac{\overline{F}\left(  G^{\leftarrow}\left(
1-s\overline{G}\left(  a_{k}\right)  \right)  \right)  }{\overline{F}\left(
G^{\leftarrow}\left(  1-\overline{G}\left(  a_{k}\right)  \right)  \right)
}.
\]
Since $s\rightarrow\overline{F}\circ G^{\leftarrow}\left(  1-s\right)  $ is
regularly varying at infinity with index $\gamma_{2}/\gamma,$ then, from
Proposition B.1.10 in \cite{deHF06}, we have for all large $n$%
\begin{equation}
\omega_{n}\left(  s\right)  :=\left\vert \ell_{n}\left(  s\right)
-s^{\gamma_{2}/\gamma}\right\vert \leq\epsilon s^{\gamma_{2}/\gamma\pm
\epsilon}. \label{ln}%
\end{equation}
Recall that $x_{0}>0$ is fixed, then $\sup_{x\geq x_{0}}\sup_{0<s\leq
x^{-1/\gamma_{2}}}\omega_{n}\left(  s\right)  \rightarrow0,$ as $n\rightarrow
\infty.$ On the other hand, by using Levy's modulus of continuity of the
Wiener process \citep[see, e.g., Theorem 1.1.1 in][]{CR-81}, we have%
\[
\left\vert \mathbf{W}\left(  \ell_{n}\left(  s\right)  \right)  -\mathbf{W}%
\left(  s^{\gamma_{2}/\gamma}\right)  \right\vert \leq2\sqrt{\omega_{n}\left(
s\right)  \log\left(  1/\omega_{n}\left(  s\right)  \right)  },
\]
uniformly on $s\geq x^{-1/\gamma_{2}},$ almost surely. By using the fact that,
$\log u<\epsilon u^{-\epsilon}$ as $u\downarrow0,$ together with inequality
$(\ref{ln}),$ we get $\left\vert \mathbf{W}\left(  \ell_{n}\left(  s\right)
\right)  -\mathbf{W}\left(  s^{\gamma_{2}/\gamma}\right)  \right\vert
\leq2\epsilon s^{\left(  \gamma_{2}/\gamma\right)  \left(  1-\epsilon\right)
/2}.$ Following our convention, we may write that $\left(  \gamma_{2}%
/\gamma\pm\epsilon\right)  \left(  1-\epsilon/2\right)  \equiv\gamma
_{2}/\gamma\pm\epsilon.$ Since $\gamma_{1}<\gamma_{2}$ then $\gamma
_{2}/\left(  2\gamma\right)  >1$ and after elementary calculation, we show
that uniformly on $x\geq x_{0}$%
\[
\frac{\gamma}{\gamma_{1}}x^{1/\gamma_{1}}\int_{0}^{x^{-1/\gamma_{2}}}%
s^{-2}\mathbf{W}\left(  \ell_{n}\left(  s\right)  \right)  ds=\frac{\gamma
}{\gamma_{1}}x^{1/\gamma_{1}}\int_{0}^{x^{-1/\gamma_{2}}}s^{-2}\mathbf{W}%
\left(  s^{\gamma_{2}/\gamma}\right)  ds+O_{\mathbf{p}}\left(  \epsilon
\right)  x^{1/\left(  2\gamma\right)  \pm\epsilon}.
\]
By similar arguments, we get%
\[
\frac{\gamma}{\gamma_{1}}x^{1/\gamma}\mathbf{W}\left(  \frac{n}{k}\overline
{F}\left(  a_{k}x\right)  \right)  =\frac{\gamma}{\gamma_{1}}x^{1/\gamma
}\mathbf{W}\left(  x^{-1/\gamma}\right)  +O_{\mathbf{p}}\left(  \epsilon
\right)  x^{1/\left(  2\gamma\right)  \pm\epsilon}.
\]
It is obvious that $O_{\mathbf{p}}\left(  \epsilon\right)  x^{1/\left(
2\gamma\right)  \pm\epsilon}+O_{\mathbf{p}}\left(  \epsilon\right)  x^{\left(
1-\eta\right)  /\gamma\pm\epsilon}=O_{\mathbf{p}}\left(  \epsilon\right)
x^{\left(  1-\eta\right)  /\gamma\pm\epsilon},$ it follows that
\[
\frac{\sqrt{k}S_{n1}\left(  x\right)  }{\Lambda\left(  a_{k}x\right)  }%
=\frac{\gamma}{\gamma_{1}}x^{1/\gamma}\mathbf{W}\left(  x^{-1/\gamma}\right)
+\frac{\gamma}{\gamma_{1}}x^{1/\gamma_{1}}\int_{0}^{x^{-1/\gamma_{2}}}%
s^{-2}\mathbf{W}\left(  s^{\gamma_{2}/\gamma}\right)  ds+O_{\mathbf{p}}\left(
\epsilon\right)  x^{\left(  1-\eta\right)  /\gamma\pm\epsilon}.
\]
After a change of variables, this may be rewritten into%
\begin{align}
&  \frac{\sqrt{k}S_{n1}\left(  x\right)  }{\Lambda\left(  a_{k}x\right)
}\label{Sn1}\\
&  =\frac{\gamma}{\gamma_{1}}x^{1/\gamma}\mathbf{W}\left(  x^{-1/\gamma
}\right)  +\frac{\gamma}{\gamma_{1}+\gamma_{2}}x^{1/\gamma}\int_{0}%
^{1}t^{-\gamma/\gamma_{2}-1}\mathbf{W}\left(  x^{-1/\gamma}t\right)
dt+O_{\mathbf{p}}\left(  \epsilon\right)  x^{\left(  1-\eta\right)  /\gamma
\pm\epsilon}.\nonumber
\end{align}
Now, we consider the second term $S_{n2}\left(  x\right)  .\ $We have
$\overline{F}_{n}\left(  z\right)  =0,$ for $z\geq X_{n:n},$ thus%
\[
S_{n2}\left(  x\right)  =\int_{xX_{n-k:n}}^{X_{n:n}}\frac{C_{n}\left(
z\right)  -C\left(  z\right)  }{C_{n}\left(  z\right)  C\left(  z\right)
}d\overline{F}_{n}\left(  z\right)  .
\]
Therefore%
\[
\left\vert S_{n2}\left(  x\right)  \right\vert \leq\theta_{n}\int_{xX_{n-k:n}%
}^{\infty}\frac{\left\vert C_{n}\left(  z\right)  -C\left(  z\right)
\right\vert }{C^{2}\left(  z\right)  }dF_{n}\left(  z\right)  ,
\]
where $\theta_{n}:=\sup_{X_{1:n}\leq z\leq X_{n:n}}\left\{  C\left(  z\right)
/C_{n}\left(  z\right)  \right\}  ,$ which is stochastically bounded
\citep[see, e.g.,][]{sw2008}. By recalling that $C=\overline{G}-\overline{F}$
and $C_{n}=\overline{G}_{n}-\overline{F}_{n},$ with $G_{n}$ denoting the
empirical df of $G,$ we write $\left\vert S_{n2}\left(  x\right)  \right\vert
\leq\theta_{n}\left(  T_{n1}\left(  x\right)  +T_{n2}\left(  x\right)
\right)  ,$ where%
\[
T_{n1}\left(  x\right)  :=\int_{xX_{n-k:n}}^{\infty}\dfrac{\left\vert
\overline{F}_{n}\left(  z\right)  -\overline{F}\left(  z\right)  \right\vert
}{C^{2}\left(  z\right)  }dF_{n}\left(  z\right)
\]
and%
\[
T_{n2}\left(  x\right)  :=\int_{xX_{n-k:n}}^{\infty}\dfrac{\left\vert
\overline{G}_{n}\left(  z\right)  -\overline{G}\left(  z\right)  \right\vert
}{C^{2}\left(  z\right)  }dF_{n}\left(  z\right)  .
\]
It is easy to verify that, by a change of variables, we have%
\begin{align*}
\frac{\sqrt{k}T_{n1}\left(  x\right)  }{\Lambda\left(  a_{k}x\right)  }  &
=d_{n}\left(  x\right)  \frac{k/n}{C\left(  a_{k}\right)  }\frac{C\left(
a_{k}x\right)  }{C\left(  a_{k}\right)  }\\
&  \times\frac{C^{2}\left(  a_{k}\right)  }{C^{2}\left(  xX_{n-k:n}\right)
}\int_{1}^{\infty}\dfrac{\left\vert \alpha_{n}\left(  \dfrac{n}{k}\overline
{F}\left(  xX_{n-k:n}z\right)  \right)  \right\vert }{C^{2}\left(
xX_{n-k:n}z\right)  /C^{2}\left(  xX_{n-k:n}\right)  }d\frac{F_{n}\left(
xX_{n-k:n}z\right)  }{\overline{F}\left(  a_{k}\right)  }.
\end{align*}
Recall that, uniformly on $x\geq x_{0},$ we have $C\left(  a_{k}\right)
/C\left(  xX_{n-k:n}\right)  =O_{\mathbf{p}}\left(  1\right)  x^{1/\gamma
_{2}\pm\epsilon}.$ Moreover, we use $\left(  \ref{C_inequa1}\right)  $ and
$\left(  \ref{dn2}\right)  $ to write%
\begin{align*}
\frac{\sqrt{k}T_{n1}\left(  x\right)  }{\Lambda\left(  a_{k}x\right)  }  &
=O_{\mathbf{p}}\left(  \frac{k/n}{C\left(  a_{k}\right)  }\right)
x^{1/\gamma\pm\epsilon}\\
&  \times\int_{1}^{\infty}z^{2/\gamma_{2}}\left\vert \alpha_{n}\left(
\dfrac{n}{k}\overline{F}\left(  xX_{n-k:n}z\right)  \right)  \right\vert
d\frac{F_{n}\left(  xX_{n-k:n}z\right)  }{\overline{F}\left(  a_{k}\right)  }.
\end{align*}
\ On the other hand, by using the stochastic boundedness of $\sup_{0<s\leq
1}\left\vert \alpha_{n}\left(  s\right)  \right\vert /s^{\eta}$ we get%
\[
\frac{\sqrt{k}T_{n1}\left(  x\right)  }{\Lambda\left(  a_{k}x\right)
}=O_{\mathbf{p}}\left(  \frac{k/n}{C\left(  a_{k}\right)  }\right)  x^{\left(
1-\eta\right)  /\gamma\pm\epsilon}\int_{1}^{\infty}z^{2/\gamma_{2}-\eta
/\gamma\pm\epsilon}d\frac{\overline{F}_{n}\left(  xX_{n-k:n}z\right)
}{\overline{F}\left(  a_{k}\right)  },
\]
where the integral may be split as follows%
\[
\int_{1}^{\infty}z^{2/\gamma_{2}-\eta/\gamma\pm\epsilon}d\frac{\overline
{F}_{n}\left(  xX_{n-k:n}z\right)  }{\overline{F}\left(  a_{k}\right)  }%
=P_{n}\left(  x\right)  +Q_{n}\left(  x\right)  ,
\]
where%
\[
P_{n}\left(  x\right)  :=\int_{1}^{\infty}z^{2/\gamma_{2}-\eta/\gamma
\pm\epsilon}d\left\{  \frac{\overline{F}_{n}\left(  xX_{n-k:n}z\right)
-\overline{F}\left(  xX_{n-k:n}z\right)  }{\overline{F}\left(  a_{k}\right)
}\right\}  ,
\]
and%
\[
Q_{n}\left(  x\right)  :=\int_{1}^{\infty}z^{2/\gamma_{2}-\eta/\gamma
\pm\epsilon}d\frac{\overline{F}\left(  xX_{n-k:n}z\right)  }{\overline
{F}\left(  a_{k}\right)  }.
\]
It is clear that
\[
P_{n}\left(  x\right)  =k^{-1/2}\int_{1}^{\infty}z^{2/\gamma_{2}-\eta
/\gamma\pm\epsilon}d\alpha_{n}\left(  \frac{n}{k}\overline{F}\left(
xX_{n-k:n}z\right)  \right)  .
\]
By similar arguments as those used above, we show that $P_{n}\left(  x\right)
=o_{\mathbf{p}}\left(  x^{-\eta/\gamma\pm\epsilon}\right)  $ and $Q_{n}\left(
x\right)  =O_{\mathbf{p}}\left(  x^{-1/\gamma\pm\epsilon}\right)  .$ Therefore%
\[
\frac{\sqrt{k}T_{n1}\left(  x\right)  }{\Lambda\left(  a_{k}x\right)
}=x^{-\eta/\gamma\pm\epsilon}O_{\mathbf{p}}\left(  \frac{k/n}{C\left(
a_{k}\right)  }\right)  .
\]
Next, let $V_{i}:=\overline{G}\left(  Y_{i}\right)  ,$ $i=1,...,n,$ and define
the corresponding tail empirical process $\beta_{n}\left(  s\right)
:=\sqrt{k}\left(  \mathbf{V}_{n}\left(  s\right)  -s\right)  ,$ for $0\leq
s\leq1,$ where $\mathbf{V}_{n}\left(  s\right)  :=k^{-1}\sum_{i=1}%
^{n}\mathbf{1}\left(  V_{i}<ks/n\right)  .$ Like for $\alpha_{n}\left(
s\right)  ,$ we also have $\sup_{0<s\leq1}\left\vert \beta_{n}\left(
s\right)  \right\vert /s^{\eta}=O_{\mathbf{p}}\left(  1\right)  ,$ therefore
by similar arguments as those used for $T_{n1}\left(  x\right)  ,$ with the
facts that $\overline{G}\left(  t\right)  \sim C\left(  t\right)  $ as
$t\rightarrow\infty$ and $\gamma_{2}>\gamma,$ we show that
\[
\frac{\sqrt{k}T_{n2}\left(  x\right)  }{\Lambda\left(  a_{k}x\right)
}=O_{\mathbf{p}}\left(  \frac{k/n}{C^{1-\eta}\left(  a_{k}\right)  }\right)
x^{\left(  1-\eta\right)  /\gamma\pm\epsilon}.
\]
From Lemma $\ref{Lemma1}$ $\left(  ii\right)  $, we have that both $\dfrac
{n}{k}C\left(  a_{k}\right)  $ and $\dfrac{n}{k}C^{1-\eta}\left(
a_{k}\right)  $ tend to infinity, it follows that
\[
\frac{\sqrt{k}T_{n1}\left(  x\right)  }{\Lambda\left(  a_{k}x\right)
}=o_{\mathbf{p}}\left(  x^{-\eta/\gamma\pm\epsilon}\right)  \text{ and }%
\frac{\sqrt{k}T_{n2}\left(  x\right)  }{\Lambda\left(  a_{k}x\right)
}=o_{\mathbf{p}}\left(  x^{\left(  1-\eta\right)  /\gamma\pm\epsilon}\right)
.
\]
Since $o_{\mathbf{p}}\left(  x^{-\eta/\gamma\pm\epsilon}\right)
+o_{\mathbf{p}}\left(  x^{\left(  1-\eta\right)  /\gamma\pm\epsilon}\right)
=o_{\mathbf{p}}\left(  x^{\left(  1-\eta\right)  /\gamma\pm\epsilon}\right)
,$ then%
\begin{equation}
\frac{\sqrt{k}S_{n2}\left(  x\right)  }{\Lambda\left(  a_{k}x\right)
}=o_{\mathbf{p}}\left(  x^{\left(  1-\eta\right)  /\gamma\pm\epsilon}\right)
. \label{Sn2}%
\end{equation}
Let us now focus on the third term $S_{n3},$ which, by integration by parts,
equals the sum of%
\[
S_{n3}^{\left(  1\right)  }\left(  x\right)  :=-\int_{xX_{n-k:n}}^{xa_{k}%
}\frac{\overline{F}_{n}\left(  z\right)  -\overline{F}\left(  z\right)
}{C^{2}\left(  z\right)  }dC\left(  z\right)  ,
\]
and
\[
S_{n3}^{\left(  2\right)  }\left(  x\right)  =-\frac{\overline{F}_{n}\left(
a_{k}x\right)  -\overline{F}\left(  a_{k}x\right)  }{C\left(  a_{k}x\right)
}+\frac{\overline{F}_{n}\left(  xX_{n-k:n}\right)  -\overline{F}\left(
xX_{n-k:n}\right)  }{C\left(  xX_{n-k:n}\right)  }.
\]
By using the change of variables $z=txa_{k}$ we get%
\[
\frac{\sqrt{k}S_{n3}^{\left(  1\right)  }\left(  x\right)  }{\Lambda\left(
a_{k}x\right)  }=-d_{n}\left(  x\right)  \int_{X_{n-k:n}/a_{k}}^{1}%
\frac{\alpha_{n}\left(  \dfrac{n}{k}\overline{F}\left(  a_{k}xz\right)
\right)  }{\left(  C\left(  a_{k}xz\right)  /C\left(  a_{k}x\right)  \right)
^{2}}d\frac{C\left(  a_{k}xz\right)  }{C\left(  a_{k}x\right)  },
\]
and%
\[
\frac{\sqrt{k}S_{n3}^{\left(  2\right)  }\left(  x\right)  }{\Lambda\left(
a_{k}x\right)  }=-d_{n}\left(  x\right)  \left\{  \alpha_{n}\left(  \frac
{n}{k}\overline{F}\left(  a_{k}x\right)  \right)  -\frac{C\left(
a_{k}x\right)  }{C\left(  xX_{n-k:n}\right)  }\alpha_{n}\left(  \dfrac{n}%
{k}\overline{F}\left(  xX_{n-k:n}\right)  \right)  \right\}  .
\]
Routine manipulations, including Proposition B.1.10 in \cite{deHF06} and the
stochastic boundedness of $\sup_{0<s\leq1}\left\vert \alpha_{n}\left(
s\right)  \right\vert /s^{\eta},$ yield%
\[
\frac{\sqrt{k}S_{n3}^{\left(  1\right)  }\left(  x\right)  }{\Lambda\left(
a_{k}x\right)  }=o_{\mathbf{p}}\left(  x^{\left(  1-\eta\right)  /\gamma
\pm\epsilon}\right)  \text{ and }\frac{\sqrt{k}S_{n3}^{\left(  2\right)
}\left(  x\right)  }{\Lambda\left(  a_{k}x\right)  }=o_{\mathbf{p}}\left(
x^{\left(  1-\eta\right)  /\gamma\pm\epsilon}\right)  .
\]
It follows that
\begin{equation}
\sqrt{k}S_{n3}\left(  x\right)  /\Lambda\left(  a_{k}x\right)  =o_{\mathbf{p}%
}\left(  x^{\left(  1-\eta\right)  /\gamma\pm\epsilon}\right)  . \label{Sn3}%
\end{equation}
By gathering results $\left(  \ref{Sn1}\right)  ,$ $\left(  \ref{Sn2}\right)
$ and $\left(  \ref{Sn3}\right)  ,$ we obtain%
\begin{align}
&  \sqrt{k}\frac{\Lambda_{n}\left(  xX_{n-k:n}\right)  -\Lambda\left(
xX_{n-k:n}\right)  }{\Lambda\left(  a_{k}x\right)  }\label{Lambda}\\
&  =\frac{\gamma}{\gamma_{1}}x^{1/\gamma}\mathbf{W}\left(  x^{-1/\gamma
}\right)  +\frac{\gamma}{\gamma_{1}+\gamma_{2}}x^{1/\gamma}\int_{0}%
^{1}t^{-\gamma/\gamma_{2}-1}\mathbf{W}\left(  x^{-1/\gamma}t\right)
dt+O_{\mathbf{p}}\left(  \epsilon\right)  x^{\left(  1-\eta\right)  /\gamma
\pm\epsilon},\nonumber
\end{align}
which yields that%
\begin{align*}
&  x^{1/\gamma_{1}}\sqrt{k}\mathbf{M}_{n1}^{\ast}\left(  x\right) \\
&  =x^{1/\gamma}\left\{  \frac{\gamma}{\gamma_{1}}\mathbf{W}\left(
x^{-1/\gamma}\right)  +\frac{\gamma}{\gamma_{1}+\gamma_{2}}\int_{0}%
^{1}t^{-\gamma/\gamma_{2}-1}\mathbf{W}\left(  x^{-1/\gamma}t\right)
dt\right\}  +O_{\mathbf{p}}\left(  \epsilon\right)  x^{\left(  1-\eta\right)
/\gamma\pm\epsilon}.
\end{align*}
We show that the expectation of the absolute value of the first term in the
right-hand side of the previous equation equals $O_{\mathbf{p}}\left(
x^{1/\left(  2\gamma\right)  }\right)  .$ Since $1/\left(  2\gamma\right)
<\left(  1-\eta\right)  /\gamma,$ we have $x^{1/\gamma_{1}}\sqrt{k}%
\mathbf{M}_{n1}^{\ast}\left(  x\right)  =O_{\mathbf{p}}\left(  x^{\left(
1-\eta\right)  /\gamma\pm\epsilon}\right)  ,$ which leads to%
\[
x^{1/\gamma_{1}}\sqrt{k}\mathbf{M}_{n1}\left(  x\right)  =x^{1/\gamma_{1}%
}\sqrt{k}\mathbf{M}_{n1}^{\ast}\left(  x\right)  +o_{\mathbf{p}}\left(
x^{\left(  1-\eta\right)  /\gamma\pm\epsilon}\right)  .
\]
Recall that $\epsilon>0$ is chosen sufficiently small, then for any
$0<\eta<1/2,$ we have%
\begin{align*}
&  x^{1/\gamma_{1}}\sqrt{k}\mathbf{M}_{n1}\left(  x\right) \\
&  =x^{1/\gamma}\left\{  \frac{\gamma}{\gamma_{1}}\mathbf{W}\left(
x^{-1/\gamma}\right)  +\frac{\gamma}{\gamma_{1}+\gamma_{2}}\int_{0}%
^{1}t^{-\gamma/\gamma_{2}-1}\mathbf{W}\left(  x^{-1/\gamma}t\right)
dt\right\}  +O_{\mathbf{p}}\left(  \epsilon\right)  x^{\left(  1-\eta\right)
/\gamma\pm\epsilon}.
\end{align*}
Before we treat the term $\mathbf{M}_{n2}\left(  x\right)  ,$ it is worth
mentioning that by letting $x=1$ in the previous approximation, we infer that%
\begin{equation}
\frac{\overline{\mathbf{F}}_{n}\left(  X_{n-k:n}\right)  }{\overline
{\mathbf{F}}\left(  X_{n-k:n}\right)  }-1=O_{\mathbf{p}}\left(  k^{-1/2}%
\right)  =o_{\mathbf{p}}\left(  1\right)  ,\text{ }\left(  k\rightarrow
\infty\right)  . \label{important}%
\end{equation}
This, with the regular variation of $\overline{\mathbf{F}},$ imply that%
\begin{equation}
\frac{\overline{\mathbf{F}}\left(  xX_{n-k:n}\right)  }{\overline{\mathbf{F}%
}_{n}\left(  X_{n-k:n}\right)  }=\left(  1+O_{\mathbf{p}}\left(
x^{\pm\epsilon}\right)  \right)  x^{-1/\gamma_{1}}. \label{important2}%
\end{equation}
To represent $\sqrt{k}\mathbf{M}_{n2}\left(  x\right)  ,$ we apply results
$\left(  \ref{Lambda}\right)  $ (for $x=1)$ and $\left(  \ref{important2}%
\right)  $ to get%
\[
x^{1/\gamma_{1}}\sqrt{k}\mathbf{M}_{n2}\left(  x\right)  =-\left\{
\frac{\gamma}{\gamma_{1}}\mathbf{W}\left(  1\right)  +\frac{\gamma}{\gamma
_{1}+\gamma_{2}}\int_{0}^{1}t^{-\gamma/\gamma_{2}-1}\mathbf{W}\left(
t\right)  dt\right\}  +O_{\mathbf{p}}\left(  \epsilon\right)  x^{\pm\epsilon
}.
\]
For the third term $\mathbf{M}_{n3}\left(  x\right)  ,$ we write%
\[
x^{1/\gamma_{1}}\sqrt{k}\mathbf{M}_{n3}\left(  x\right)  =\left(
\frac{\overline{\mathbf{F}}\left(  xX_{n-k:n}\right)  }{\overline{\mathbf{F}%
}_{n}\left(  X_{n-k:n}\right)  }-x^{-1/\gamma_{1}}\right)  x^{1/\gamma_{1}%
}\sqrt{k}\mathbf{M}_{n1}\left(  x\right)  ,
\]
which, by equation $\left(  \ref{important2}\right)  ,$ is equal to
$O_{\mathbf{p}}\left(  \epsilon\right)  x^{-1/\gamma_{1}+\left(
1-\eta\right)  /\gamma\pm\epsilon}.$ Let $\eta_{0}$ be such that
$\gamma/\gamma_{2}<\eta_{0}<\eta<1/2,$ then $\eta_{0}-\eta<0$ and for
$\epsilon>0$ sufficiently small, we have $\left(  \eta_{0}-\eta\right)
/\gamma+\epsilon<0.$ Since $x\geq x_{0}>0,$ then $O_{\mathbf{p}}\left(
\epsilon\right)  x^{\left(  \eta_{0}-\eta\right)  /\gamma\pm\epsilon
}=O_{\mathbf{p}}\left(  \epsilon\right)  $ and thus%
\begin{equation}
x^{1/\gamma_{1}-\left(  1-\eta_{0}\right)  /\gamma}\left\{  \sqrt{k}\left(
\mathbf{M}_{n1}\left(  x\right)  +\mathbf{M}_{n2}\left(  x\right)
+\mathbf{M}_{n3}\left(  x\right)  \right)  -\mathbf{\Gamma}\left(
x;\mathbf{W}\right)  \right\}  =O_{\mathbf{p}}\left(  \epsilon\right)  ,
\label{MM}%
\end{equation}
where $\mathbf{\Gamma}\left(  x;\mathbf{W}\right)  $ is the Gaussian process
given in Theorem \ref{Theorem1}. For the fourth term $\mathbf{M}_{n4}\left(
x\right)  ,$ it suffices to use the uniform inequality to second-order
condition $\left(  \ref{second-order}\right)  ,$ given in assertion (2.3.23)
of Theorem 2.3.9 in \cite{deHF06}, to get
\[
\sqrt{k}\mathbf{M}_{n4}\left(  x\right)  =\left(  1+o_{\mathbf{p}}\left(
1\right)  \right)  x^{-1/\gamma_{1}}\dfrac{x^{\tau_{1}/\gamma_{1}}-1}%
{\gamma_{1}\tau_{1}}\sqrt{k}\widetilde{\mathbf{A}}_{\mathbf{F}}\left(
1/\overline{\mathbf{F}}\left(  X_{n-k:n}\right)  \right)  ,
\]
for a possibly different function $\widetilde{\mathbf{A}}_{\mathbf{F}}$ with
$\widetilde{\mathbf{A}}_{\mathbf{F}}\sim\mathbf{A}_{\mathbf{F}}.$ Then
Proposition B.1.10 in \cite{deHF06} and the fact that $t\rightarrow
\widetilde{\mathbf{A}}_{\mathbf{F}}\left(  1/\overline{\mathbf{F}}\left(
\mathbb{U}_{F}\left(  t\right)  \right)  \right)  =:\mathbf{A}_{0}$ $\left(
t\right)  $ is regularly varying with index $\tau_{1}/\gamma_{1}$ with
$X_{n-k:n}/a_{k}\overset{\mathbf{p}}{\rightarrow}1,$ imply that%
\[
\frac{\mathbf{A}_{0}\left(  1/\overline{F}\left(  X_{n-k:n}\right)  \right)
}{\mathbf{A}_{0}\left(  1/\overline{F}\left(  a_{k}\right)  \right)  }%
\overset{\mathbf{p}}{\rightarrow}1,\text{ as }n\rightarrow\infty,
\]
as well. Since $o_{\mathbf{p}}\left(  1\right)  x^{-1/\gamma_{1}}%
\dfrac{x^{\tau_{1}/\gamma_{1}}-1}{\gamma_{1}\tau_{1}}=o_{\mathbf{p}}\left(
x^{-1/\gamma_{1}+\left(  1-\eta\right)  /\gamma\pm\epsilon}\right)  ,$ and by
assumption $\sqrt{k}\mathbf{A}_{0}\left(  1/\overline{F}\left(  a_{k}\right)
\right)  =\sqrt{k}\mathbf{A}_{0}\left(  n/k\right)  =O\left(  1\right)  ,$ it
follows that
\[
\sqrt{k}\mathbf{M}_{n4}\left(  x\right)  =x^{-1/\gamma_{1}}\dfrac{x^{\tau
_{1}/\gamma_{1}}-1}{\gamma_{1}\tau_{1}}\sqrt{k}\mathbf{A}_{0}\left(
n/k\right)  +o_{\mathbf{p}}\left(  x^{-1/\gamma_{1}+\left(  1-\eta\right)
/\gamma\pm\epsilon}\right)  .
\]
Finally, by letting $\epsilon\downarrow0$ in $\left(  \ref{MM}\right)  ,$ we
end up with%
\[
\sup_{x\geq x_{0}}x^{1/\gamma_{1}-\left(  1-\eta_{0}\right)  /\gamma
}\left\vert \mathbf{D}_{n}\left(  x\right)  -\mathbf{\Gamma}\left(
x;\mathbf{W}\right)  -x^{-1/\gamma_{1}}\dfrac{x^{\tau_{1}/\gamma_{1}}%
-1}{\gamma_{1}\tau_{1}}\sqrt{k}\mathbf{A}_{0}\left(  n/k\right)  \right\vert
\overset{\mathbf{P}}{\rightarrow}0,
\]
for every $x_{0}>0$\ and $\gamma/\gamma_{2}<\eta_{0}<\eta<1/2.$\textbf{\ }%
Letting $\eta_{0}:=1/2-\xi$ and recalling that $1/\gamma_{1}=1/\gamma
-1/\gamma_{2}$ yields that $0<\xi<1/2-\gamma/\gamma_{2}$\ and achieves the
proof.$\hfill\square$

\subsection{Proof of Theorem \ref{Theorem2}}

We start by proving the consistency of $\widehat{\gamma}_{1}$ that we write as
$\widehat{\gamma}_{1}=\int_{1}^{\infty}x^{-1}\overline{\mathbf{F}}_{n}\left(
xX_{n-k:n}\right)  /\overline{\mathbf{F}}_{n}\left(  X_{n-k:n}\right)  dx.$ It
is readily checked that this may be decomposed into the sum of
\[
I_{1n}:=\int_{1}^{\infty}x^{-1}\frac{\overline{\mathbf{F}}\left(
xX_{n-k:n}\right)  }{\overline{\mathbf{F}}\left(  X_{n-k:n}\right)  }dx\text{
and }I_{2n}:=\int_{1}^{\infty}x^{-1}\sum_{i=1}^{3}\mathbf{M}_{ni}\left(
x\right)  dx.
\]
By the regular variation of $\overline{\mathbf{F}}$ $\left(  \ref{RV-1}%
\right)  $ and Potter's inequalities, we get $I_{1n}\overset{\mathbf{P}%
}{\rightarrow}\gamma_{1}$ as $n\rightarrow\infty.$ Then, we just need to show
that $I_{2n}$ tends to zero in probability. From $\left(  \ref{MM}\right)  $
we have%
\[
I_{2n}=\frac{1}{\sqrt{k}}\int_{1}^{\infty}x^{-1}\mathbf{\Gamma}\left(
x;\mathbf{W}\right)  dx+\frac{1}{\sqrt{k}}\int_{1}^{\infty}x^{-1}%
o_{\mathbf{p}}\left(  x^{\left(  1-\eta\right)  /\gamma-1/\gamma_{1}}\right)
dx.
\]
On the one hand, since $\gamma/\gamma_{2}<\eta,$ the second integral above is
finite and therefore the second term of $I_{2n}$ is negligible in probability.
On the other hand, we have%
\begin{align*}
\int_{1}^{\infty}x^{-1}\mathbf{\Gamma}\left(  x;\mathbf{W}\right)  dx  &
=\frac{\gamma}{\gamma_{1}}\int_{1}^{\infty}x^{1/\gamma_{2}-1}\left\{
\mathbf{W}\left(  x^{-1/\gamma}\right)  -x^{-1/\gamma}\mathbf{W}\left(
1\right)  \right\}  dx+\frac{\gamma}{\gamma_{1}+\gamma_{2}}\\
&  \times\int_{1}^{\infty}x^{1/\gamma_{2}-1}\left\{  \int_{0}^{1}%
s^{-\gamma/\gamma_{2}-1}\left\{  \mathbf{W}\left(  x^{-1/\gamma}s\right)
-x^{-1/\gamma}\mathbf{W}\left(  s\right)  \right\}  ds\right\}  dx,
\end{align*}
which, after some elementary but tedious manipulations of integral calculus
(change of variables and integration by parts), becomes%
\begin{align}
\int_{1}^{\infty}x^{-1}\mathbf{\Gamma}\left(  x;\mathbf{W}\right)  dx  &
=-\gamma\mathbf{W}\left(  1\right) \label{GAMMA}\\
&  +\frac{\gamma}{\gamma_{1}+\gamma_{2}}\int_{0}^{1}\left(  \gamma_{2}%
-\gamma_{1}-\gamma\log s\right)  s^{-\gamma/\gamma_{2}-1}\mathbf{W}\left(
s\right)  ds.\nonumber
\end{align}
By using the facts that $\mathbf{E}\left\vert \mathbf{W}\left(  s\right)
\right\vert \leq s^{1/2}$ and $\gamma_{1}<\gamma_{2},$ we deduce that
$\int_{1}^{\infty}x^{-1}\mathbf{\Gamma}\left(  x;\mathbf{W}\right)  dx$ is
stochastically bounded and therefore the first term of $I_{2n}$ is is
negligible in probability as well. Consequently, we have $I_{2n}%
=o_{\mathbf{P}}\left(  1\right)  $ when $n\rightarrow\infty,$ as sought. As
for the Gaussian representation result, it is easy to verify that $\sqrt
{k}\left(  \widehat{\gamma}_{1}-\gamma_{1}\right)  =\int_{1}^{\infty}%
x^{-1}\mathbf{D}_{n}\left(  x\right)  dx.$ Then, applying Theorem
\ref{Theorem1} yields that%
\[
\sqrt{k}\left(  \widehat{\gamma}_{1}-\gamma_{1}\right)  =\frac{\sqrt
{k}\mathbf{A}_{0}\left(  n/k\right)  }{1-\tau}+\int_{1}^{\infty}%
x^{-1}\mathbf{\Gamma}\left(  x;\mathbf{W}\right)  dx+o_{\mathbf{P}}\left(
1\right)  ,
\]
and finally, using result $\left(  \ref{GAMMA}\right)  $ completes the
proof.$\hfill\square$

\subsection{Proof of Corollary \ref{Cor1}}

We set%
\[
\sqrt{k}\left(  \widehat{\gamma}_{1}-\gamma_{1}\right)  =\gamma\Delta
+\frac{\sqrt{k}\mathbf{A}_{0}\left(  n/k\right)  }{1-\tau}+o_{\mathbf{P}%
}\left(  1\right)  ,
\]
where $\Delta:=a\Delta_{1}+b\Delta_{2}-\Delta_{3},$ with $a:=\left(
\gamma_{2}-\gamma_{1}\right)  /\left(  \gamma_{1}+\gamma_{2}\right)  ,$
$b:=-\gamma/\left(  \gamma_{1}+\gamma_{2}\right)  $ and%
\[
\Delta_{1}:=\int_{0}^{1}s^{\rho-2}\mathbf{W}\left(  s\right)  ds,\text{
}\Delta_{2}:=\int_{0}^{1}s^{\rho-2}\mathbf{W}\left(  s\right)  \log sds,\text{
}\Delta_{3}:=\mathbf{W}\left(  1\right)  ,
\]
with $\rho:=1-\gamma/\gamma_{2}>0.$

\noindent It is clear that the asymptotic mean is equal to $\lim
_{n\rightarrow\infty}\sqrt{k}\mathbf{A}_{0}\left(  n/k\right)  /\left(
1-\tau\right)  ,$ while for the asymptotic variance we find, after elementary
but tedious computations, the following covariances:%
\begin{align*}
\mathbf{E}\left[  \Delta_{1}^{2}\right]   &  =\frac{2}{\rho\left(
2\rho-1\right)  },\text{ }\mathbf{E}\left[  \Delta_{2}^{2}\right]
=\frac{2\left(  4\rho-1\right)  }{\rho^{2}\left(  2\rho-1\right)  ^{3}},\text{
}\mathbf{E}\left[  \Delta_{3}^{2}\right]  =1,\\
\mathbf{E}\left[  \Delta_{1}\Delta_{2}\right]   &  =\frac{1-4\rho}{\rho
^{2}\left(  2\rho-1\right)  ^{2}},\text{ }\mathbf{E}\left[  \Delta_{1}%
\Delta_{3}\right]  =\frac{1}{\rho},\text{ }\mathbf{E}\left[  \Delta_{2}%
\Delta_{3}\right]  =-\frac{1}{\rho^{2}}.
\end{align*}
It follows that
\[
\mathbf{E}\left[  \Delta^{2}\right]  =\frac{2a^{2}}{\rho\left(  2\rho
-1\right)  }+\frac{2b^{2}\left(  4\rho-1\right)  }{\rho^{2}\left(
2\rho-1\right)  ^{3}}+\frac{2ab\left(  1-4\rho\right)  }{\rho^{2}\left(
2\rho-1\right)  ^{2}}+\frac{2b}{\rho^{2}}-\frac{2a}{\rho}+1.
\]
Replacing $a,$ $b$ and $\rho$ by their values achieves the proof.$\hfill
\square\medskip$

\noindent\textbf{Concluding notes\medskip}

\noindent We would like to emphasize the fact that, unlike \cite{GS2015} who
defined their estimator in terms of two (not necessarily equal) sample
fractions $k=k^{\prime}$ of upper order statistics from $X$ and $Y$
respectively, we build our estimator on the basis of just a single sample
fraction. The consideration of two distinct sample fractions poses a problem
from a computational point of view, as the issue of selecting an optimal
couple of sample fractions is not as easy and usual as determining just one
best number of top statistics to be used in parameter estimate computation.
Besides that, \cite{GS2015} didn't treat the asymptotic normality when
$k=k^{\prime}$ and only carried out their simulations in this very particular
case, as they mentioned in their conclusion.\textbf{ }For these reasons, we
don't compare the two estimators in Section \ref{sec5}.\textbf{\medskip}

\noindent A more thorough simulation study, with confidence interval
construction and eventual comparison with the estimator of \cite{GS2015}, will
be part of a future work. Another point, beyond the scope of the present
paper, that deserves to be considered is to reduce estimation biases under
random truncation. Similar anterior works were done with complete datasets by,
for instance, \cite{PengQi-04}, \cite{LiPeng-10} and \cite{BMNY-13}%
.\textbf{\medskip}

\noindent We finish this work by making a comment on relation $\left(
\ref{log}\right)  ,$ which actually is a special case of a more general
functional of the distribution tail defined by%
\[
\Gamma_{t}\left(  g,\alpha\right)  :=\frac{\dfrac{1}{\overline{F}\left(
t\right)  }\int_{t}^{\infty}g\left(  \dfrac{\overline{F}\left(  x\right)
}{\overline{F}\left(  t-\right)  }\right)  \left(  \log\dfrac{x}{t}\right)
^{\alpha}dF\left(  x\right)  }{\int_{0}^{1}g\left(  x\right)  \left(  -\log
x\right)  ^{\alpha}dx},\text{ }t\geq0,
\]
where $g$ is some weight function and $\alpha$ some positive real number. As a
consequence of the fact that $\lim_{t\rightarrow\infty}\Gamma_{t}\left(
g,\alpha\right)  =\gamma^{\alpha},$ this functional can be considered as the
starting point to constructing a whole class of estimators for distribution
tail parameters. Indeed, in the complete data case, we replace $F$ by its
empirical counterpart $F_{n}$ and $t$ by $X_{n-k:n}$ to get the following
statistic which generalizes several extreme value theory based procedures of
estimation already existing in the literature:%
\[
\Gamma_{n,k}\left(  g,\alpha\right)  :=\frac{\dfrac{1}{k}\sum\limits_{i=1}%
^{k}g\left(  \dfrac{i}{k+1}\right)  \left(  \log\dfrac{X_{n-i+1:n}}{X_{n-k:n}%
}\right)  ^{\alpha}}{\int_{0}^{1}g\left(  x\right)  \left(  -\log x\right)
^{\alpha}dx}.
\]
When $g=\alpha=1,$ we recover the famous Hill estimator \citep{Hill75}. For a
detailed list of extreme value index estimators drawn from the statistic
above, we refer to the paper of \cite{Mercadier-10}, where the authors propose
an estimation approach of the second-order parameter by considering
differences and quotients of several forms of $\Gamma_{n,k}\left(
g,\alpha\right)  .$ By analogy, when we deal with randomly truncated
observations, we substitute the product-limit estimator $\mathbf{F}_{n}$ for
$F$ in the formula of $\Gamma_{t}\left(  g,\alpha\right)  $ in order to obtain
the following family of parameter estimators under random truncation:%
\[
\Gamma_{n,k}\left(  g,\alpha\right)  :=\frac{\sum\limits_{i=1}^{k}%
a_{n}^{\left(  i\right)  }g\left(  \dfrac{\overline{\mathbf{F}}_{n}\left(
X_{n-i+1:n}\right)  }{\overline{\mathbf{F}}_{n}\left(  X_{n-k-1:n}\right)
}\right)  \left(  \log\dfrac{X_{n-i+1:n}}{X_{n-k:n}}\right)  ^{\alpha}}%
{\sum\limits_{i=1}^{k}a_{n}^{\left(  i\right)  }\int_{0}^{1}g\left(  x\right)
\left(  -\log x\right)  ^{\alpha}dx}.
\]
This would have fruitful consequences on the statistical analysis of extremes
under random truncation.

\section{\textbf{Appendix}}

\begin{lemma}
\textbf{\label{Lemma0}}Assume that both second-order conditions
$(\ref{second-order})$\ and $(\ref{second-orderG})$\ hold. Then, for all large
$x,$ there exist constants $d_{1},d_{2}>0,$ such that%
\[
\overline{F}\left(  x\right)  =\left(  1+o\left(  1\right)  \right)
d_{1}x^{-1/\gamma}\text{ and }\overline{G}\left(  x\right)  =\left(
1+o\left(  1\right)  \right)  d_{2}x^{-1/\gamma_{2}}.
\]

\end{lemma}

\begin{proof}
We only show the first statement since the second one follows by similar
arguments. To this end, we rewrite the first equation of $\left(
\ref{FbarGbar}\right)  $ into%
\[
\overline{F}\left(  x\right)  =-p^{-1}\overline{\mathbf{G}}\left(  x\right)
\overline{\mathbf{F}}\left(  x\right)  \int_{1}^{\infty}\frac{\overline
{\mathbf{G}}\left(  xz\right)  }{\overline{\mathbf{G}}\left(  x\right)
}d\frac{\overline{\mathbf{F}}\left(  xz\right)  }{\overline{\mathbf{F}}\left(
x\right)  }.
\]
By applying Proposition B.1.10 in \cite{deHF06} to both $\overline{\mathbf{F}%
}$ and $\overline{\mathbf{G}},$ it is easy to check that%
\[
\int_{1}^{\infty}\frac{\overline{\mathbf{G}}\left(  xz\right)  }%
{\overline{\mathbf{G}}\left(  x\right)  }d\frac{\overline{\mathbf{F}}\left(
xz\right)  }{\overline{\mathbf{F}}\left(  x\right)  }=-\left(  1+o\left(
1\right)  \right)  \gamma/\gamma_{1}.
\]
On the other hand, since $\overline{\mathbf{F}}$ and $\overline{\mathbf{G}}$
satisfy the aforementioned second-order conditions, then in view of Lemma 3 in
\cite{HJ-2011}, there exist two constants $a_{1},a_{2}>0,$ such that
$\overline{\mathbf{F}}\left(  x\right)  =\left(  1+o\left(  1\right)  \right)
a_{1}x^{-1/\gamma_{1}}$ and $\overline{\mathbf{G}}\left(  x\right)  =\left(
1+o\left(  1\right)  \right)  a_{2}x^{-1/\gamma_{2}},$ as $x\rightarrow
\infty.$ Therefore $\overline{F}\left(  x\right)  =\left(  1+o\left(
1\right)  \right)  d_{1}x^{-1/\gamma}$ with $d_{1}=p^{-1}a_{1}a_{2}%
\gamma/\gamma_{1}.$
\end{proof}

\begin{lemma}
\textbf{\label{Lemma1}}Under the assumptions of Lemma $\ref{Lemma0},$ we have%
\[%
\begin{tabular}
[c]{l}%
$\left(  i\right)  $ $\lim\limits_{t\rightarrow\infty}C\left(  t\right)
/\overline{G}\left(  t\right)  =1.$\\
$\left(  ii\right)  \text{ }\lim\limits_{t\rightarrow\infty}t^{1/\nu}C\left(
\mathbb{U}_{F}\left(  t\right)  \right)  =\infty,$ for each $0<\nu\leq1.$\\
$\left(  iii\right)  $ $\lim_{t\rightarrow\infty}\sup_{x\geq x_{0}%
}x^{-1/\gamma\pm\epsilon}\left\vert \left(  t\Lambda\left(  x\mathbb{U}%
_{F}\left(  t\right)  \right)  C\left(  x\mathbb{U}_{F}\left(  t\right)
\right)  \right)  ^{-1}-\left(  \gamma/\gamma_{1}\right)  x^{1/\gamma
}\right\vert =0,$\\
\ \ \ \ \ \ for $x_{0}>0$ and any sufficiently small $\epsilon>0.$%
\end{tabular}
\ \ \ \
\]

\end{lemma}

\begin{proof}
For assertion $\left(  i\right)  ,$ write $C\left(  t\right)  /\overline
{G}\left(  t\right)  =1-\overline{F}\left(  t\right)  /\overline{G}\left(
t\right)  $ and observe that from Lemma $\ref{Lemma0}$ we have $\overline
{F}\left(  t\right)  /\overline{G}\left(  t\right)  =\left(  1+o\left(
1\right)  \right)  \left(  d_{1}/d_{2}\right)  t^{1/\gamma_{2}-1/\gamma}.$
Since $1/\gamma_{2}-1/\gamma<0,$ then $\overline{F}\left(  t\right)
/\overline{G}\left(  t\right)  =o\left(  1\right)  ,$ that is $C\left(
t\right)  /\overline{G}\left(  t\right)  =1+o\left(  1\right)  $ as sought.
For result $\left(  ii\right)  ,$ Lemma $\ref{Lemma0}$ implies that
$\mathbb{U}_{F}\left(  t\right)  =\left(  1+o\left(  1\right)  \right)
\left(  d_{1}t\right)  ^{\gamma}$ (as $t\rightarrow\infty$)$,$ it follows that
$C\left(  \mathbb{U}_{F}\left(  t\right)  \right)  =\left(  1+o\left(
1\right)  \right)  d_{2}\left(  d_{1}t\right)  ^{-\gamma/\gamma_{2}}.$ Since
$0<\gamma/\gamma_{2}<1,$ then for every $0<\nu\leq1,$ $t^{1/\nu}C\left(
\mathbb{U}_{F}\left(  t\right)  \right)  \rightarrow\infty$ as $t\rightarrow
\infty.$ To prove $\left(  iii\right)  ,$ we first show that
\begin{equation}
t\Lambda\left(  x\mathbb{U}_{F}\left(  t\right)  \right)  C\left(
x\mathbb{U}_{F}\left(  t\right)  \right)  -\left(  \gamma_{1}/\gamma\right)
x^{-1/\gamma}=O\left(  \epsilon\right)  x^{-1/\gamma\pm\epsilon}.
\label{gamma}%
\end{equation}
Recalling that $\Lambda\left(  x\right)  =$ $\int_{x}^{\infty}dF\left(
z\right)  /C\left(  z\right)  $ and $\overline{F}\left(  \mathbb{U}_{F}\left(
t\right)  \right)  =1/t,$ we write%
\begin{equation}
t\Lambda\left(  x\mathbb{U}_{F}\left(  t\right)  \right)  C\left(
x\mathbb{U}_{F}\left(  t\right)  \right)  =-\frac{C\left(  x\mathbb{U}%
_{F}\left(  t\right)  \right)  }{C\left(  \mathbb{U}_{F}\left(  t\right)
\right)  }\int_{x}^{\infty}\frac{C\left(  \mathbb{U}_{F}\left(  t\right)
\right)  }{C\left(  z\mathbb{U}_{F}\left(  t\right)  \right)  }\frac
{d\overline{F}\left(  z\mathbb{U}_{F}\left(  t\right)  \right)  }{\overline
{F}\left(  \mathbb{U}_{F}\left(  t\right)  \right)  }. \label{formula1}%
\end{equation}
Observe now that $t\Lambda\left(  x\mathbb{U}_{F}\left(  t\right)  \right)
C\left(  x\mathbb{U}_{F}\left(  t\right)  \right)  -\dfrac{\gamma_{1}}{\gamma
}x^{-1/\gamma}$ may be decomposed into the sum of%
\[
D_{1}\left(  s;t\right)  :=-\left(  \frac{C\left(  x\mathbb{U}_{F}\left(
t\right)  \right)  }{C\left(  \mathbb{U}_{F}\left(  t\right)  \right)
}-x^{-1/\gamma_{2}}\right)  \int_{x}^{\infty}\frac{C\left(  \mathbb{U}%
_{F}\left(  t\right)  \right)  }{C\left(  z\mathbb{U}_{F}\left(  t\right)
\right)  }\frac{d\overline{F}\left(  z\mathbb{U}_{F}\left(  t\right)  \right)
}{\overline{F}\left(  \mathbb{U}_{F}\left(  t\right)  \right)  },
\]%
\[
D_{2}\left(  s;t\right)  :=-x^{-1/\gamma_{2}}\int_{x}^{\infty}\left(
\frac{C\left(  \mathbb{U}_{F}\left(  t\right)  \right)  }{C\left(
z\mathbb{U}_{F}\left(  t\right)  \right)  }-z^{1/\gamma_{2}}\right)
\frac{d\overline{F}\left(  z\mathbb{U}_{F}\left(  t\right)  \right)
}{\overline{F}\left(  \mathbb{U}_{F}\left(  t\right)  \right)  }%
\]
and%
\[
D_{3}\left(  s;t\right)  :=-x^{-1/\gamma_{2}}\int_{x}^{\infty}z^{1/\gamma_{2}%
}d\left(  \frac{\overline{F}\left(  z\mathbb{U}_{F}\left(  t\right)  \right)
}{\overline{F}\left(  \mathbb{U}_{F}\left(  t\right)  \right)  }-z^{-1/\gamma
}\right)  .
\]
By applying Proposition B.1.10 in \cite{deHF06} to both $C$ and $\overline{F}$
with integrations by parts, it is easy to verify that
\[
\left\vert t\Lambda\left(  x\mathbb{U}_{F}\left(  t\right)  \right)  C\left(
x\mathbb{U}_{F}\left(  t\right)  \right)  -\left(  \gamma_{1}/\gamma\right)
x^{-1/\gamma}\right\vert \leq\epsilon x^{-1/\gamma\pm\epsilon}.
\]
Observe now that $t\Lambda\left(  x\mathbb{U}_{F}\left(  t\right)  \right)
C\left(  x\mathbb{U}_{F}\left(  t\right)  \right)  -\left(  \gamma/\gamma
_{1}\right)  x^{1/\gamma}$ is equal to%
\[
\left(  \left(  t\Lambda\left(  x\mathbb{U}_{F}\left(  t\right)  \right)
C\left(  x\mathbb{U}_{F}\left(  t\right)  \right)  \right)  ^{-1}\right)
^{-1}-\left(  \left(  \gamma_{1}/\gamma\right)  x^{-1/\gamma}\right)  ^{-1}.
\]
By using the mean value theorem, the latter equals%
\[
\frac{\left(  \gamma_{1}/\gamma\right)  x^{-1/\gamma}-t\Lambda\left(
x\mathbb{U}_{F}\left(  t\right)  \right)  C\left(  x\mathbb{U}_{F}\left(
t\right)  \right)  }{\left(  \psi\left(  x;t\right)  \right)  ^{2}},
\]
where $\psi\left(  x;t\right)  $ is between $\left(  \gamma_{1}/\gamma\right)
x^{-1/\gamma}$ and $t\Lambda\left(  x\mathbb{U}_{F}\left(  t\right)  \right)
C\left(  x\mathbb{U}_{F}\left(  t\right)  \right)  .$ In view of the
representation $\left(  \ref{formula1}\right)  $ and Potter's inequalities,
applied to $C$ and $\overline{F},$ with an integration by parts, we get
$t\Lambda\left(  x\mathbb{U}_{F}\left(  t\right)  \right)  C\left(
x\mathbb{U}_{F}\left(  t\right)  \right)  \geq\left(  1-\epsilon\right)
x^{-1/\gamma\pm\epsilon}.$ It follows that $\left(  \psi\left(  x;t\right)
\right)  ^{2}\geq\left(  1-\epsilon\right)  ^{2}x^{-2/\gamma\pm2\epsilon}$ and
therefore%
\[
\left\vert t\Lambda\left(  x\mathbb{U}_{F}\left(  t\right)  \right)  C\left(
x\mathbb{U}_{F}\left(  t\right)  \right)  -\left(  \gamma/\gamma_{1}\right)
x^{1/\gamma}\right\vert \leq\left(  1-\epsilon\right)  ^{-2}\epsilon
x^{1/\gamma\pm\epsilon},
\]
as sought.
\end{proof}

\end{document}